\documentclass[11pt,reqno]{amsart}

\usepackage{geometry}

\usepackage{graphicx, amssymb}
\usepackage{hyperref, mathrsfs}
\usepackage{mathtools}
\usepackage{combelow}

\usepackage{hyperref}

\usepackage{tikz, tkz-tab, tikz-cd}

\newtheorem{theorem}{Theorem}[section]
\newtheorem{prop}[theorem]{Proposition}
\newtheorem{lemma}[theorem]{Lemma}
\newtheorem{cor}[theorem]{Corollary}
\newtheorem{question}[theorem]{Question}

\newtheorem*{notation*}{Notation}

\newtheorem*{convention*}{Convention}
\newtheorem{conj}{Conjecture}[section]

\newtheorem{mainthm}{Theorem}

\theoremstyle{definition}
\newtheorem{definition}[theorem]{Definition}
\newtheorem{remark}[theorem]{Remark}
\newtheorem{example}[theorem]{Example}
\numberwithin{equation}{section}

\newcommand{\K}{\mathbf{k}}
\newcommand{\N}{\mathbb{N}}

\newcommand{\wt}{\widetilde}

\newcommand{\iv}{\operatorname{iv}}
\newcommand{\Tor}{\operatorname{Tor}}
\newcommand{\Ext}{\operatorname{Ext}}

\newcommand{\pd}{\operatorname{pd}}

\newcommand{\depth}{\operatorname{depth}}
\newcommand{\reg}{\operatorname{reg}}

\newcommand{\diam}{\operatorname{diam}}

\usetikzlibrary{decorations.markings}
\tikzstyle{vertex}=[circle, draw, inner sep=0pt, minimum size=5pt]

\usepackage[T1]{fontenc}
\usepackage{mathabx}

\usepackage[theoremfont]{newtxtext}

\begin{document}
\title[On binomial edge ideals of corona of graphs]{On binomial edge ideals of corona of graphs}
\author{Buddhadev Hajra}
\address{Chennai Mathematical Institute\newline \indent H1 SIPCOT IT Park, Siruseri, Kelambakkam 603103, India}
\email{hajrabuddhadev92@gmail.com}

\author{Rajib Sarkar}
\address{Department of Mathematics, National Institute of Technology Warangal \newline \indent Hanamkonda, Telangana-- 506004, India.}
\email{rajibs@nitw.ac.in/rajib.sarkar63@gmail.com}

\begin{abstract}
For a simple graph $G$, let $J_G$ denote the corresponding binomial edge ideal. This article considers the binomial edge ideal of the corona product of two connected graphs $G$ and $H$. The corona product of $G$ and $H$, denoted by $G\circ H$, is a construction where each vertex of $G$ is connected (via the coning-off) to an entire copy of $H$. This is a direct generalization of a cone construction. Previous studies have shown that for $J_{G \circ H}$ to be Cohen-Macaulay, both $G$ and $H$ must be complete graphs. However, there are no general formulae for the dimension, depth, or Castelnuovo-Mumford regularity of $J_{G\circ H}$ for all graphs $G$ and $H$. In this article, we provide a general formula for the dimension, depth and Castelnuovo-Mumford regularity of the binomial edge ideals of certain corona and corona-type (somewhat a generalization of corona) products of special interests. Additionally, we study the Cohen-Macaulayness, unmixedness and related properties of binomial edge ideals corresponding to above class of graphs. We have also added a short note on the reduction of the Bolognini-Macchia-Strazzanti Conjecture to all graphs with a diameter of $3$. 
\end{abstract}
\dedicatory{Dedicated to the memory of J\"urgen Herzog}
\maketitle
\tableofcontents
	
    {\let\thefootnote\relax\footnote{{{\bf{2020 Mathematics Subject Classification.}} 05E40, 13C13, 13D02}\\
	{\bf{Data availability statement.}} Not applicable.}}

\section{\bf Introduction}

In the past two decades, the dynamic intersection of combinatorics of a graph and the algebra of various ideals arising from it has become a vibrant area of research. One of the most extensively studied ideals in this context is the (monomial) edge ideal. Studying the Cohen-Macaulayness property of edge ideals is an important topic in this area. Though the complete classification of all simple graphs having Cohen-Macaulay edge ideal is not yet known, an exhaustive classification of bipartite graphs corresponding to the Cohen-Macaulay edge ideals is well understood. 

In 2010, Herzog et al. in \cite{HHHKR10} and Ohtani independently in \cite{Oht11} introduced the concept of binomial edge ideals for simple graphs --- the graphs having no loops or multiple edges. We refer the readers to the book \cite{HHO18} and \cite{Das22, SM18} for a detailed exploration of this topic. Consider a simple graph $G$ with the vertex set $V(G)=[n]:=\{1,2,\ldots,n\}$ and the edge set $E(G)$. Let $S = \K[x_1,\ldots,x_n,y_1,\ldots,y_n]$ be the polynomial ring in $2n$ variables over an arbitrary ﬁeld $\K$. Then, the binomial edge ideal corresponding to the graph $G$ denoted by $J_G$ is an ideal that is deﬁned as follows:
$$J_G = \left\langle m_{ij}: \{i, j\} \in E(G), 1 \leq i < j \leq n\right\rangle \subseteq S,$$ where $m_{ij}$ denotes the $(i,j)$-th $(2\times 2)$-minor $\det\begin{pmatrix} 
    x_i & x_j \\ 
    y_i & y_j 
\end{pmatrix}
=x_i y_j - x_j y_i$
of the following $(2\times n)$-matrix $$M_G:=
\begin{pmatrix} 
    x_1 & x_2 & \cdots  & x_n\\
    y_1 & y_2 & \cdots  & y_n\\
\end{pmatrix}.$$
The class of binomial edge ideals serves as a natural generalization of the \emph{determinantal ideal} $I_2$ generated by all the $(2 \times 2)$-minors of the $(2 \times n)$-matrix $M_G$ of indeterminates defined above. In particular, when $G$ is the complete graph $K_n$, the binomial edge ideal $J_G$ of $G$ coincides with the ideal $I_2$. 

Over the years, many researchers have paid attention to characterizing Cohen-Macaulay binomial edge ideals in the same spirit as in the case of classical edge ideals. Various endeavors have been made to address this challenge across different interesting families of graphs, viz., chordal, bipartite, (generalized) block, traceable, Hamiltonian, unicyclic, etc. Important contributions in this pursuit include the research in \cite{BMRS24,BMS18,BMS22,EHH11,LMRR23,LRR24,RR14,Rin13,Rin19,SMK18,SS22,SS24}. Specifically, in \cite{BMS22}, the authors introduced insightful combinatorial properties of Cohen-Macaulay binomial edge ideals, a contribution later leveraged in \cite{LMRR23} to computationally classify all such ideals for graphs with up to $12$ number of vertices. Recently, in \cite{BMRS24}, the authors developed an algorithm to characterize Cohen-Macaulay binomial edge ideals for the graphs with up to $15$ number of vertices. Despite these efforts, achieving a comprehensive classification of graphs with Cohen-Macaulay binomial edge ideals appears to be a formidable task. However, researchers have delved into a fundamental algebraic invariant, namely, the Castelnuovo-Mumford regularity of binomial edge ideals, employing the combinatorial intricacies of the underlying graph. This line of investigation has proven fruitful in the classification of specific Cohen-Macaulay binomial edge ideals, as evidenced in \cite{ERT20,RS23,SMK18}. The exploration of both algebraic properties and numerical invariants of binomial edge ideals has received very significant attention in recent times.


Throughout this article, unless explicitly stated otherwise, we assume the graphs under discussion are connected. The \emph{corona product} of two graphs, denoted as $G\circ H$, is defined as follows: take one copy of graph $G$ and $|V(G)|$ copies of graph $H$, connecting each vertex from the $i$-th copy of $H$ with the $i$-th vertex of the graph $G$. The notation $H_v$ refers to the copy of $H$ connected to vertex $v$ in $G\circ H$, and we use this notation quite often in the subsequent section. In particular, when $H=K_1$ (a \emph{point graph}), $G\circ H$ is equivalent to the \emph{whisker graph}, the graph which is obtained from a graph $G$ by adding a whisker to each of its vertices. This fascinating construction of the corona of two graphs first appeared in \cite{FH70}. As mentioned in \cite{FH70}, it immediately
follows from the definition of the corona that
\begin{center}
    $|V(G\circ H)|=|V(G)|(1+|V(H)|)$;\\[0.25cm]
    $|E(G\circ H)|=|E(G)|+|V(G)|(|V(H)|+|E(H)|)$.
\end{center}
Simple examples illustrate that, in general, $G\circ H$ is not equivalent to $H\circ G$. Furthermore, the associative law does not hold; i.e., $G_1 \circ(G_2 \circ G_3)$ and $(G_1 \circ G_2) \circ G_3$ represent distinct graphs for any three graphs $G_1, G_2, G_3$. 

Since the class of corona of any two graphs appears to be interesting, it is important to check the Cohen-Macaulayness of the binomial edge ideal of the corona product. It can be easily observed by studying the cutsets that very often, the binomial edge ideal of a corona product fails to be Cohen-Macaulay. Therefore it is evident that in the case when $J_{G\circ H}$ is indeed Cohen-Macaulay, the two graphs $G$ and $H$ are rather very special. In \cite{KSM15}, Kiani and Saeedi Madani briefly studied the corona product of two graphs. They investigated the strong correlation between unmixedness and the Cohen-Macaulayness properties of the binomial edge ideals associated with the corona of two graphs. They proved that in the case of the corona product of two graphs, the two notions --- unmixedness and Cohen-Macaulayness, in fact, are equivalent. One of the key results in this direction is --- ``\emph{for any two graphs $G$ and $H$ with $|V(G)|\geq2$, if $J_{G\circ H}$ is Cohen-Macaulay (equivalently, unmixed), both the graphs $G$ and $H$ have to be complete graphs}'' (cf. \cite[Theorem 5.4]{KSM15}). Since the unmixedness of a binomial edge ideal corresponding to any graph can be completely understood from the combinatorial data of the underlying graph, in the paper \cite{KSM15}, the authors didn't dig into the derivation of any formula for the well-known algebraic invariants such as depth, Castelnuovo-Mumford regularity, etc. of the binomial edge ideal corresponding to the corona product in order to study the Cohen-Macaulayness of the same.

A general formula for depth and Castelnuovo-Mumford regularity for the binomial edge ideals of the corona product purely in terms of those invariants of the two individual graphs would still be a very interesting but rather challenging question. In this article, we give a general formula for depth and Castelnuovo-Mumford regularity for the affine $\K$-algebra $S/J_{B\circ H}$, where $B$ is a Cohen-Macaulay closed graph and $H$ is an arbitrary graph.

\begin{notation*}{\rm 
Let $G$ be a graph on $V(G)$. We reserve the notation $S$ for the polynomial ring $\K[x_u, y_u : u \in V(G)]$ and for $v\in V(G)$, $S'$ for the polynomial ring $\K[x_u,y_u : u\in V(G)\setminus \{v\}]$.  If $H$ is any other graph with the vertex set $V(H)$, then we set $S_H = \K[x_u, y_u : u \in V(H)]$ and for $w\in V(H)$, set $S_H'=\K[x_u,y_u : u\in V(H)\setminus \{w\}]$.
}
\end{notation*}

\begin{convention*}{\rm
To be consistent, we reserve the following nomenclature throughout this article. A graph $G$ is said to be unmixed (resp. Cohen-Macaulay) if $J_G$, the corresponding binomial edge ideal, is unmixed (resp. Cohen-Macaulay).
}
\end{convention*}
\subsection{Structure and main results of the paper}\mbox{}

We have added the necessary preliminaries, both combinatorial and algebraic, in Section 2. The results obtained in the two subsequent sections (viz. Section 2 \& Section 3) are of two different flavors. Firstly, we build the combinatorial understanding of unmixed/accessible/Cohen-Macaulay binomial edge ideals of the corona product and related graphs (see Section 3). We conclude Section 3 by adding a short note on the reduction of Bolognini-Macchia-Strazzanti Conjecture (B-M-S Conjecture, in short) to all simple graphs with diameter 3 (see Question \ref{Ques: Reduction of B-M-S Conjecture}). Section 4 is a bit technical and is solely devoted to the computation of dimension, depth, and regularity formulae for certain binomial edge ideals of special interest.

We prove the following main results:

\begin{mainthm}[Theorems \ref{Thm: Unmixedness of L-corona of G and H implies H is unmixed}, \ref{Thm: Acc. system of L-corona of G and H implies H has acc. system}]\label{thm:mainA}
    Let $G$ and $H$ be two graphs such that $|V(G)|>1$. For a non-empty subset $L$ of $V(G)$, if $J_{G\circ_L H}$ is unmixed (resp. accessible), then $J_H$ is also unmixed (resp. accessible).
\end{mainthm}
The converse to the above result is not true in general. For this, we gave an example.\\

The (generalized) $L$-corona product of any graph with the complete graph, i.e., the graphs $K_n\circ_\ell H$ for all $n \geq 2$ and any graph $H$ are of special interest. We prove the following results:

\begin{mainthm}[{Theorem \ref{Thm: Combined -- l-corona with complete is unmixed, accessible, CM iff the other is unmixed, accessible, CM respectively}}]\label{thm:mainB}
    Let $H$ be any graph. For integers $n$ and $\ell$ such that $n\geq 2$ and $1\leq \ell<n$, $K_n\circ_\ell H$ is unmixed (or accessible or Cohen-Macaulay) if and only if $H$ is unmixed (or accessible or Cohen-Macaulay, respectively).
\end{mainthm}
\begin{mainthm}[{Theorem \ref{Thm: Dimension of l-corona with complete graph}}]\label{thm:mainC}
    Let $H$ be any connected graph and $n\geq 1$ any integer. Consider the graph $G:=K_n\circ_\ell H$ with $1\leq \ell \leq n$. Then 
    \begin{equation*}
        \dim(S/J_G)= n-\ell +1+\ell \dim (S_H/J_H),
    \end{equation*}
    where $|V(H)|= h$. Moreover, if $J_H$ is unmixed then
    \begin{equation*}
        \dim(S/J_{K_n\circ H})=n+nh+1
    \end{equation*}
\end{mainthm}

This dimension formula extends the result in \cite[Corollary 5.2]{KSM15}.
\begin{mainthm}[{Theorem \ref{Thm: Depth for corona with complete graph}}]\label{thm:mainD}
For any connected graph $H$ and any integer $n\geq 1$, let $G:=K_n\circ H$. Then the following assertions hold.
    \begin{enumerate}
        \item $\depth(S/J_G)=
        \begin{cases}
            1+n\cdot\depth(S_H/J_H), & \text{whenever } H \text{ is complete};\\
            n\cdot\depth(S_H/J_H), & \text{whenever } H \text{ is non-complete}.
        \end{cases}
        $\\
        \item $\reg(S/J_G)=
        \begin{cases}
            n+1, & \text{whenever } H \text{ is complete};\\
            n\cdot\reg(S_H/J_H), & \text{whenever } H \text{ is non-complete}.
        \end{cases}
        $
        \item Assume that $H$ is not a complete graph, $\pd(S_H/J_H)=p_H$ and $\beta_{p_H,p_H+r_H}(S_H/J_H)$ is an extremal Betti number of $S_H/J_H$ for $r_H\geq 2$. Then $\beta_{p,p+nr_H}(S/J_G)$ is an extremal Betti number if $n=2$ and $\beta_{p,p+nr_H+1}(S/J_G)$ is an extremal Betti number if $n\geq 3$, where $p=\pd(S/J_G)=2n+np_H.$
    \end{enumerate}
\end{mainthm}

More generally, we prove the following:

\begin{mainthm}[{Theorem \ref{Thm: Depth for corona with CM closed graph}}]\label{thm:mainE}
Let $B$ be a Cohen-Macaulay closed graph and $H$ any connected graph with $|V(B)|=b$. Consider the graph $G:=B\circ H$ for $n \geq 1$. Then 
\begin{enumerate}
    \item $\depth(S/J_G)=
\begin{cases}
1+b \cdot \depth(S_H/J_H), &\text{if } H \text{ is complete};\\
b\cdot \depth(S_H/J_H), &\text{otherwise}.
\end{cases}
$
    \item $\reg(S/J_G)=
        \begin{cases}
            b+1, & \text{if } H \text{ is complete};\\
            b\cdot\reg(S_H/J_H), & \text{otherwise }.
        \end{cases}
        $
        \item Assume that $H$ is not a complete graph, $\pd(S_H/J_H)=p_H$ and $\beta_{p_H,p_H+r_H}(S_H/J_H)$ is an extremal Betti number of $S_H/J_H$ for $r_H\geq 2$. Then $\beta_{p,p+br_H+1}(S/J_G)$ is an extremal Betti number, where $p=\pd(S/J_G)=2b+bp_H.$
\end{enumerate}
\end{mainthm}

\section{\bf Preliminaries}
\subsection{Combinatorial Preliminaries}
We recall some basic notation and definitions from graph theory which will be used throughout this article.

For a graph $G$, we denote the vertex set and edge set respectively by $V(G)$ and $E(G)$. A subgraph $H$ of $G$ is said to be an \textit{induced subgraph} if for $u, v \in V(H)$, $\{u,v\} \in E(H)$ if and only if $\{u,v\} \in E(G)$. For a graph $G$, a \emph{clique} is a complete subgraph of $G$. A vertex $v \in V(G)$ is said to be a \emph{simplicial vertex} if it belongs to exactly one maximal clique of $G$ and a vertex is called \emph{internal} if it is not simplicial. The set of internal vertices of $G$ is denoted by $\iv(G)$.
Given $v \in V(G)$, define the \textit{neighborhood of }$v$ in $G$ as $N_G(v) := \{u ~:~ \{u,v\} \in E(G)\}\subset V(G)$. 
For $v \in V(G)$, let $G_v$  denote the graph with vertices $V(G)$ and $E(G_v) =  E(G) \cup \{\{u, w\} ~:~ u, w \in N_G(v)\}$.

For any subset $A\subset V(G)$, by abuse of notation, $G \setminus A$ denotes the induced subgraph of $G$ on the vertex set $V(G)\setminus A$. In particular, for a vertex $v \in V(G)$, we denote by $G\setminus v$ the induced subgraph of $G$ on the vertex set $V(G)\setminus \{v\}$.


For any subset $T \subset V(G)$, we reserve the notation $\omega(G\setminus T)$ to denote the number of connected components of the induced subgraph $G \setminus T$. A subset $T\subset V(G)$ is called a \emph{cutset} in $G$ if either $T=\emptyset$ or $T\neq \emptyset$ with $\omega(G\setminus (T\setminus \{v\}))<\omega(G\setminus T)$ for every vertex $v\in T$. The existence of a non-empty cutset is always guaranteed in every graph, which is different from a complete graph $K_n$. Let $\mathscr{C}(G)$ denote the set of all cutsets in $G$.

We will say that $\mathscr{C}(G)$ is an \emph{accessible set system} if for every non-empty $T\in \mathscr{C}(G)$, there exists $t\in T$ such that $T\setminus \{t\} \in \mathscr{C}(G)$. A graph $G$ is called \emph{accessible} if its binomial edge ideal $J_G$ is unmixed and $\mathscr{C}(G)$ is an \emph{accessible set system}. These definitions are being used in \cite{BMS22}.

\subsection{Algebraic Preliminaries}
We now recall the definition of two important homological invariants, which can be computed directly from the Betti table. One of them is the Castelnuovo-Mumford regularity, which measures the complexity of the structure of a graded module over a polynomial ring $A$. Let $M$ be a finitely generated graded $A$-module and $\beta_{i,i+j}(M)$ the graded Betti numbers. The \textit{projective dimension} of $M$ is denoted by $\pd_A(M)$ (or, simply $\pd(M)$) and is defined as $$\pd_A(M):=\max\{i : \beta_{i,i+j}(M) \neq 0 \text{ for some } j\};$$
 and the \textit{Castelnuovo-Mumford regularity} (or simply, \textit{regularity}) of $M$ is defined as 
 $$\reg(M):=\max \{j : \beta_{i,i+j}(M) \neq 0 \text{ for some } i\}.$$ 

The following formula due to Auslander and Buchsbaum is useful for our later discussion.

\begin{theorem}[Auslander--Buchsbaum Formula]
    Let $R$ be a commutative Noetherian local (or graded) ring and $M$ is a non-zero finitely generated $R$-module such that $\pd_R(M)<\infty$, then $$\pd_R(M)+\depth(M)=\depth(R).$$
\end{theorem}
For the depth computation of binomial edge ideals, we need the following important result due to Ohtani.

\begin{lemma}[{Ohtani, cf. \cite[Lemma 4.8]{Oht11}}]\label{Lem: Non-simplicial vertex SES}
    Let $G$ be a graph and $v \in V(G)$ be a non-simplicial vertex. Then 
    $$J_G = (J_{G\setminus v} + (x_v, y_v)) \cap J_{G_v}.$$
\end{lemma}
Thus, we have the following short exact sequence:
      \begin{align}\label{ohtani-ses}
    0\longrightarrow \frac{S}{J_{G}}\longrightarrow 
    \frac{S}{(x_v,y_v)+J_{G \setminus v}} \oplus \frac{S}{J_{{G_v}}}\longrightarrow \frac{S}{(x_v,y_v)+J_{G_v \setminus v}} \longrightarrow 0,
     \end{align}
  and correspondingly, the  long exact sequence of Tor modules:
       \begin{multline}\label{ohtani-tor}
            \cdots \longrightarrow \Tor_{i}^{S}\left( \frac{S}{J_G},\K\right)_{i+j}\longrightarrow \Tor_{i}^{S}\left( \frac{S}{(x_v,y_v)+J_{G\setminus v}},\K\right)_{i+j} \oplus \Tor_{i}^{S}\left(\frac{S}{J_{G_v}},\K\right)_{i+j}\\ 
            \longrightarrow \Tor_{i}^{S}\left(\frac{S}{(x_v,y_v)+J_{G_v\setminus v}},\K\right)_{i+j} \longrightarrow \Tor_{i-1}^{S}\left( \frac{S}{J_G},\K\right)_{i+j}\longrightarrow \cdots
     \end{multline}
     


We now recall the following basic properties of depth, which will be useful in the subsequent proofs. In the later proofs, we refer to this lemma as ``Depth Lemma''.
\begin{lemma}[{\bf Depth Lemma}]
   Let $S$ be a standard graded polynomial ring over a field $k$ and let $M$, $N$ and $P$ be finitely generated graded $S$-modules. If $$0\to M \xrightarrow{f} N \xrightarrow{g} P \to 0$$ is a short exact sequence with $f$, $g$ graded homomorphisms of degree zero, then
   \begin{enumerate}
       \item[(a)] $\depth_S(M) \geq \min\{\depth_S(N), \depth_S(P) + 1\}$;
       \item[(b)] $\depth_S(M) = \depth_S(P) + 1$, if $\depth_S(N) > \depth_S(P)$;
       \item[(c)] $\depth_S(M) = \depth_S(N)$, if $\depth_S(N) < \depth_S(P)$.
       \item[(d)] $\depth_S(N) \geq \min\{\depth_S(M), \depth_S(P)\}$;
       \item[(e)] $\depth_S(P) \geq \min\{\depth_S(N), \depth_S(M) - 1\}$;
       \item[(f)] $\depth_S(P) = \depth_S(N)$, if $\depth_S(N) < \depth_S(M)$.
   \end{enumerate}
\end{lemma}

Likewise, we recall the following basic properties of regularity, which we will use in our later proofs. From now onward, we refer to this lemma as ``Regularity Lemma''.
\begin{lemma}[{\bf Regularity Lemma}]
   Let $S$ be a standard graded polynomial ring over a field $k$ and let $M$, $N$ and $P$ be finitely generated graded $S$-modules. If $$0\to M \xrightarrow{f} N \xrightarrow{g} P \to 0$$ is a short exact sequence with $f$, $g$ graded homomorphisms of degree zero, then
   \begin{enumerate}
       \item[(a)] $\reg(M) \leq \max\{\reg(N), \reg(P) + 1\}$;
       \item[(b)] $\reg(M) = \reg(P) + 1$, if $\reg(N) < \reg(P)$;
       \item[(c)] $\reg(M) = \reg(N)$, if $\reg(N) > \reg(P)$.
   \end{enumerate}
\end{lemma}

The above two lemmas can be easily derived from the long exact sequence of $\Tor$ and $\Ext$ corresponding to the short exact sequence mentioned.

\section{\bf Unmixedness, Cohen-Macaulayness, and Accessibility of Certain Corona Products}
Let $G$ and $H$ be two graphs with $|V(G)|=m$ and $|V(H)|=n$. Let $V(G)=\{v_1,\ldots,v_m\}$. Suppose, $H_{v_i}$ (or simply $H_i$) denotes the copy of $H$ which is joined to the vertex $v_i$ of $G$ in $G\circ H$ for all $i=1,\ldots, m$. Thus $$G\circ H:=G\cup (\bigcup_{i=1}^{m}v_i\ast H_i).$$

Modifying the above, we define for a non-empty subset $L \subset V(G)$
$$G\circ_{L}H:=G\cup (\bigcup_{v\in L}v\ast H_v).$$ We call $G\circ_L H$ the (generalized) \emph{$L$-corona product} of $G$ with $H$. The definition of $G\circ_{L}H$ heavily depends on the choice of $L$. Clearly, $G\circ_{L}H$ coincides with the (usual) corona product $G\circ H$ if and only if $L=V(G)$.

Whenever $G=K_m$, all subsets $L$ of $V(G)$ with $|L|=\ell$ behave the same and thus we use the notation $K_m\circ_\ell H$ for $K_m\circ_L H$ to emphasize the fact that this is independent of the choice of $L$ with $|L|=\ell$.

Every cutset $T$ of $G\circ H$ can be expressed as $T = T_0\cup(\bigcup_{i=1}^{m}T_i)$, where $T_0 \subset V(G)$ and $T_i \subset V(H_i)$ for all $1\leq i \leq m$. The detailed description of a cutset for $G\circ H$ or, more generally, for $G\circ_{L} H$ with $L\subset V(G)$ can be understood from the result below.

\begin{prop}\label{Prop: cutset of corona}
Let $G$ and $H$ be two graphs and $L$ a non-empty proper subset of $V(G)$. Let $\emptyset \neq T \in \mathscr{C}(G\circ_{L} H)$ with $T=T_0\cup(\bigcup_{v \in L}T_v)$ for $T_0 \subset V(G)$ and $T_v \subset V(H_v)$ for all $v\in L$. Then, the following assertions hold.
\begin{enumerate}
    \item $\emptyset \neq T_0 \subsetneq V(G)$.
    \item If $v\in L\setminus T_0$, then $T_v=\emptyset$.
    \item If $v\in T_0$, then either $T_v=\emptyset$ or $T_v\in \mathscr{C}(H_v)$. 
    \item If $v\in L\cap T_0$ such that $N_G(v)\subset T_0$, then $T_v\neq \emptyset$.
    \item Let $N:=\{v\in L : T_v\neq \emptyset\}$. Then $$\omega(G\circ_L H \setminus T)=\omega(G\setminus T_0)+\sum\limits_{v\in N}\omega(H_v\setminus T_v)+|T_0\cap L|-|N|.$$
    \item If $(T_0)_{G-sim}$ denotes the set of simplicial vertices of $G$ in $T_0$, then $(T_0)_{G-sim} \subset L$. 
    
    \item If $T_0\cap L=\emptyset$, then $T=T_0\in \mathscr{C}(G)$ and $(T_0)_{G-sim}=\emptyset$.
\end{enumerate}
\end{prop}
\begin{proof} ---
We prove the assertions as follows.

{\it Proof of (1).} --- If $T_0$ is empty, then clearly $G\circ_L H \setminus T$ remains connected. Thus $T$ is not a cutset since $T$ is assumed to be non-empty. Now suppose $T_0=V(G)$. Since $L\subsetneq V(G)$, assume $w\in V(G)\setminus L$. Then $G\circ_L H \setminus (T\setminus \{w\})=\{w\}\cup(\bigcup_{v \in L}(H_v\setminus T_v))$ and thus $$\omega(G\circ_L H \setminus (T\setminus \{w\}))=1+\sum\limits_{v\in L}\omega(H_v\setminus T_v) > \sum\limits_{v\in L}\omega(H_v\setminus T_v)=\omega(G\circ_L H \setminus T),$$ a contradiction for $T$ being a cutset. 

{\it Proof of (2).} --- We can assume that $v\in L\setminus T_0$. Suppose, if possible, $T_v\neq \emptyset$ and assume that $w\in T_v$. Set $T^*:=T\setminus T_v$. Since $H_v\setminus T_v$ remains connected to the vertex $v$ in $G\circ_{L}H\setminus (T\setminus \{w\})$ and $v$ is not removed in $G\circ_{L}H\setminus (T\setminus \{w\})$ as $v\notin T_0$, evidently $\omega(G\circ_L H \setminus T^*)=\omega(G\circ_L H \setminus T)$. For the same reason, it follows that $$\omega(G\circ_L H \setminus (T\setminus \{w\})) = \omega(G\circ_L H \setminus T),$$
which contradicts the fact that $T \in \mathscr{C}(G\circ_{L}H)$.

{\it Proof of (3).} --- Suppose that $T_v\neq \emptyset$. Set $T^*:=T\setminus T_v$ as before. Observe that, since $v \in T_0$ $$G\circ_L H \setminus T=(H_v\setminus T_v)\sqcup \left(G\circ_{L\setminus \{v\}} H \setminus T^*\right).$$
Let $w\in T_v$. Since $T$ is a cutset of $G\circ_{L}H$, it follows that
\begin{align*}
    &\omega(G\circ_L H \setminus (T\setminus \{w\})<\omega(G\circ_L H \setminus T);\\
    \text{or,}\quad &\omega(H_v\setminus (T_v\setminus \{w\}))+ \omega(G\circ_{L\setminus \{v\}} H \setminus T^*) < \omega(H_v\setminus T_v)+ \omega(G\circ_{L\setminus \{v\}} H \setminus T^*);\\
    \text{or,}\quad & \omega(H_v\setminus (T_v\setminus \{w\})) < \omega(H_v\setminus T_v),
\end{align*}
as desired. This concludes that $T_v$ is a non-empty cutset of $H_v$.

{\it Proof of (4).} --- Assume that $v\in L \cap T_0$ with $N_G(v)\subset T_0$. Suppose, if possible, $T_v=\emptyset$. Therefore $$\omega(G\circ_LH\setminus T)=\omega(G\circ_LH\setminus (T\setminus \{v\})$$ with the description of both of these connected components remains the same except that the connected component $H_v$ in the former gets replaced by the connected component ${\rm cone}(v,H_v)$ in latter. However, this leads to a contradiction, since $T \in \mathscr{C}(G\circ_LH)$. Hence $T_v\neq \emptyset$.

{\it Proof of (5).} --- This proof follows similarly as in (3).

{\it Proof of (6).} --- Let $v\in (T_0)_{G-sim}$. Let $F$ be the unique clique in $G$ containing $v$. Suppose if possible, $v\notin L$. Thus $F$ remains the unique clique in $G\circ_LH$ containing $v$, which implies that $v$ is a simplicial vertex of $G\circ_LH$. Therefore using \cite[Proposition 2.1]{RR14}, it follows that $v$ is not contained in any cutset of $G\circ_LH$, contradicting the fact that $v\in T_0\subset T \in \mathscr{C}(G\circ_LH)$. Hence $v\in L$, as desired.


{\it Proof of (7).} --- Since $T_0\cap L=\emptyset$, it follows from (2) that $T_v=\emptyset$ for all $v \in L$, i.e., $N=\emptyset$. In other words, $T=T_0$. Let $v \in T_0$. Since $T_0=T\in \mathscr{C}(G\circ_L H)$, it follows that
\begin{align*}
    &\omega(G\circ_L H\setminus(T_0\setminus \{v\}))<\omega(G\circ_L H\setminus T_0);\\
    \text{or,}\quad &\omega(G\setminus (T_0\setminus \{v\})< \omega(G\setminus T_0). \quad (\text{using (5), as }T_0\cap L=N=\emptyset) 
\end{align*}
This proves $T_0\in \mathscr{C}(G)$.

In this case, if $(T_0)_{G-sim}$ happens to be non-empty and containing $w$, say, then $w$ should not lie in any cutset of $G$ following \cite[Proposition 2.1]{RR14}. In particular, $w\notin T_0$, a contradiction.
\end{proof}

\begin{remark}
Observe that the formula proved in (5) works for any subset $T \subset V(G\circ_L H)$, not necessarily a cutset of $G\circ_L H$.
\end{remark}

\begin{theorem}\label{Thm: Unmixedness of L-corona of G and H implies H is unmixed}
    Let $G$ and $H$ be two graphs such that $|V(G)|>1$. For a non-empty proper subset $L$ of $V(G)$, if $J_{G\circ_L H}$ is unmixed, then the following are true:
    \begin{enumerate}
        \item $L$ does not contain any cut vertex of $G$. More generally, for every subset $L_0\subset L$, $L_0\in \mathscr{C}(G\circ_LH)$, though the induced subgraph $G\setminus L_0$ of $G$ remains connected. In particular, $G\setminus L$ is connected.
        \item $J_H$ is unmixed.
    \end{enumerate}
    
\end{theorem}
\begin{proof} ---
    Suppose that $J_{G\circ_L H}$ is unmixed for $\emptyset \neq L \subset V(G)$ and $|V(G)|>1$.
    
    {\it Proof of (1).} --- On the contrary, assume that $v\in L$ is a cut vertex of $G$. Then clearly $\{v\}\in \mathscr{C}(G\circ_L H)$ and $(G\circ_L H)\setminus v=H_v\sqcup ((G\setminus v)\circ_{L\setminus \{v\}}H)$. Therefore $$\omega((G\circ_L H)\setminus v)=1+\omega(G\setminus v)\geq 3,$$ which implies that $J_{G\circ_L H}$ is not unmixed, a contradiction to the hypothesis and hence the proof of the first part of (1) follows.

    For the remaining part of (1), we use induction on $|L_0|$ for a fixed subset $L_0\subset L$. If $|L_0|=1$, the first part yields $L_0\in \mathscr{C}(G\circ_LH)$, though $G\setminus L_0$ is connected. Now assume that $G\setminus L_0'$ is connected for every $L_0'\subsetneq L_0$ with $L_0'\in \mathscr{C}(G\circ_LH)$. Let $u \in L_0$. Then the induction hypothesis says that $G\setminus (L_0\setminus \{u\})$ is connected. Thus from \ref{Prop: cutset of corona}(5), we get that 
    \begin{align}
        &\omega(G\circ_LH\setminus(L_0\setminus \{u\}))=\omega(G\setminus (L_0\setminus \{u\}))+|(L_0\setminus \{u\})\cap L|=|L_0|;\\
        \text{and}&\nonumber\\
        &\omega(G\circ_LH\setminus L_0)=\omega(G\setminus L_0)+|L_0\cap L|=\omega(G\setminus L_0)+|L_0|\geq 1+|L_0|, \quad(\text{as }L_0\subset L\subsetneq V(G)).\label{eqn: L_0 is a cutset of L-corona}
    \end{align}
    This implies that $L_0\in \mathscr{C}(G\circ_LH)$. Since $J_{G\circ_{L}H}$ is unmixed, it follows from \eqref{eqn: L_0 is a cutset of L-corona} that $$\omega(G\setminus L_0)+|L_0|=\omega(G\circ_LH\setminus L_0)=|L_0|+1,$$ whence $\omega(G\setminus L_0)=1$. This completes the proof.
    
    {\it Proof of (2).} --- Let $T_H \in \mathscr{C}(H)$. It is enough to prove that $\omega(H\setminus T_H)=|T_H|+1$. Consider $T_H$ as a subset of $H_v=H$ for some $v \in L$ and define $T:=\{v\}\cup T_H$. 
    
    {\bf Claim.} {\it $T\in \mathscr{C}(G\circ_LH)$.}
    
    {\it Proof of Claim.} --- Clearly, $(G\circ_L H)\setminus T=((G\setminus v)\circ_{L\setminus \{v\}} H) \sqcup (H_v\setminus T_H)$. Therefore, 
    \begin{equation}\label{eqn: star}
        \omega((G\circ_L H) \setminus T)=\omega(G\setminus v) + \omega(H_v\setminus T_H)\geq \omega(G\setminus v)+2, 
    \end{equation}
    as $T_H\in \mathscr{C}(H_v=H)$. It is evident that $(G\circ_LH) \setminus (T\setminus \{v\})=(G\circ_LH) \setminus T_H$ is connected. Now take $u\in T_H$. Therefore $(G\circ_L H) \setminus (T\setminus \{u\})=((G\setminus v)\circ_{L\setminus \{v\}} H) \sqcup (H_v\setminus (T_H\setminus \{u\}))$ and thus \begin{align*}
        \omega((G\circ_L H) \setminus (T\setminus \{u\}))&=\omega(G\setminus v) + \omega(H_v\setminus (T_H\setminus \{u\}))\\
        &< \omega(G\setminus v)+\omega(H_v\setminus T_H);\quad (\text{since } T_H \in \mathscr{C}(H_v=H))\\
        &=\omega((G\circ_L H) \setminus T).\quad (\text{by } \eqref{eqn: star})
    \end{align*}
    This proves the claim.
    
    Since $v\in L$, by what we proved in (1), it follows that $v$ is a non-cut vertex of $G$. Therefore $G\setminus v$ is connected and is non-empty too as $|V(G)|>1$. Using the unmixedness of $J_{G\circ_L H}$, it follows from \eqref{eqn: star},
    \begin{align*}
        \omega(H\setminus T_H)=\omega(H_v\setminus T_H) &=\omega((G\circ_LH)\setminus T) - \omega(G\setminus v),\\
        &= (|T|+1) - 1, \quad(\text{as } T\in \mathscr{C}(G\circ_LH) \text{ and } G\setminus v \text{ is connected})\\
        &=|T|,\\
        &=|T_H|+1.
    \end{align*}
    This completes the proof.
\end{proof}

\begin{remark}
    Related to Theorem \ref{Thm: Unmixedness of L-corona of G and H implies H is unmixed}, the case when $|V(G)|=1$ with $G=\{v\}$, i.e., $G\circ_{L}H$ coincides with ${\rm cone}(v,H):=v\ast H$ --- the cone of $v$ on $H$, was earlier considered by Rauf--Rinaldo in \cite[Lemma 3.3]{RR14}. Note that, unlike to our above result, $J_{{\rm cone}(v,H)}$ can be unmixed even though $J_H$ is not so --- e.g., \cite[Figure 1]{RR14}.
\end{remark}

\begin{example}\label{Ex: Not unmixed gen. corona}
This example shows that the converse of Theorem \ref{Thm: Unmixedness of L-corona of G and H implies H is unmixed} is not quite true even though both $J_G$ and $J_H$ are unmixed. It seems like there must be having a very strong assumption on the choice of the subset $L$ of $V(G)$. In Figure \ref{fig: corona}, $J_G$ is unmixed using \cite[Proposition 2]{Rin13}. As $J_H$ is a principal ideal, $J_H$ is also unmixed. However, $J_{G\circ_L H}$ is not unmixed. Indeed, for the cutset $S:=\{u,w\} \subset V(G\circ_L H)$, it follows that $\omega(G\circ_L H\setminus S)=4\neq |S|+1=3.$
\end{example}

\begin{figure}
    \begin{tikzpicture}[scale=1]
        \draw [line width=1pt] (-12,1) -- (-11,1);
        \draw [line width=1pt] (-12,1) -- (-12,0);
        \draw [line width=1pt] (-12,0) -- (-11,0);
        \draw [line width=1pt] (-11,0) -- (-11,1);
        \draw [line width=1pt] (-12,1) -- (-12,2);
        \draw [line width=1pt] (-11,2) -- (-11,1);
        
        \draw [line width=1pt] (-9,0.5) -- (-8,0.5);
        \draw [line width=1pt] (-5,1) -- (-4,1);
        \draw [line width=1pt] (-5,1) -- (-5,0);
        \draw [line width=1pt] (-5,0) -- (-4,0);
        \draw [line width=1pt] (-4,0) -- (-4,1);
        \draw [line width=1pt] (-5,1) -- (-5,2);
        \draw [line width=1pt] (-4,2) -- (-4,1);
        \draw [line width=1pt] (-5.98,-0.43) -- (-5,0);
        \draw [line width=1pt] (-5.98,-0.43) -- (-5.5,-0.97);
        \draw [line width=1pt] (-5.5,-0.97) -- (-5,0);
        \draw [line width=1pt] (-4,0) -- (-3.52,-0.97);
        \draw [line width=1pt] (-4,0) -- (-3.02,-0.47);
        \draw [line width=1pt] (-3.52,-0.97) -- (-3.02,-0.47);
        
        \begin{scriptsize}
        
            \fill (-12,1) circle (1.5pt);
            \fill (-11,1) circle (1.5pt);
            \fill (-11,0) circle (1.5pt);
            \fill (-12,0) circle (1.5pt);
            \fill (-12,2) circle (1.5pt);
            \fill (-11,2) circle (1.5pt);
            
            \draw (-12.25,1) node {$u$};
            \draw (-12.25,0) node {$v$};
            \draw (-10.75,0) node {$w$};
            \draw (-10.75,1) node {$x$};
            
            \draw (-11.5,-0.7) node {$G$};
            \draw (-10,0.5) node {$\circ_{L=\{v,w\}}$};
            \fill (-9,0.5) circle (1.5pt);
            \fill (-8,0.5) circle (1.5pt);
            \draw (-8.5,-0.7) node {$H$};
            \draw (-7,0.5) node {$=$};
            \fill (-5,1) circle (1.5pt);
            \fill (-4,1) circle (1.5pt);
            \fill (-4,0) circle (1.5pt);
            \fill (-5,0) circle (1.5pt);
            \fill (-5,2) circle (1.5pt);
            \fill (-4,2) circle (1.5pt);
            \fill (-3.52,-0.97) circle (1.5pt);
            \fill (-3.02,-0.47) circle (1.5pt);
            \fill (-5.5,-0.97) circle (1.5pt);
            \fill (-5.98,-0.43) circle (1.5pt);
            
            \draw (-5.25,0.1) node {$v$};
            \draw (-3.75,0.1) node {$w$};
            \draw (-5.25,1) node {$u$};
            \draw (-3.75,1) node {$x$};
            
            \draw (-4.5,-0.7) node {$G\circ_L H$};
        \end{scriptsize}
    \end{tikzpicture}
    \caption{The corona product}\label{fig: corona}
\end{figure}
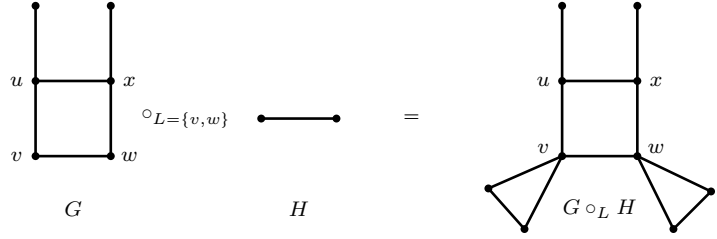
            
\begin{cor}\label{Cor: t-corona with complete is unmixed iff the other is unmixed}
    Let $H$ be any graph. For integers $n$ and $\ell$ such that $n\geq 2$ and $1\leq \ell<n$, $K_n\circ_\ell H$ is unmixed if and only if $H$ is unmixed.
\end{cor}
\begin{proof} ---
    Let $G:=K_n\circ_\ell H$ for integers $n$, $\ell$ with $n\geq 2$ and $1\leq \ell<n$ such that the $\ell$-copies of $H$ are attached to $L\subsetneq [n]=V(K_n)$. 
    
    If $J_G$ is unmixed, then Theorem \ref{Thm: Unmixedness of L-corona of G and H implies H is unmixed} implies that $J_H$ is also unmixed. 
    
    Conversely, assume that $J_H$ is unmixed. Suppose $\emptyset \neq T :=T_0\cup(\bigcup_{v \in L}T_v)\in \mathscr{C}(G)$ for some $T_0 \subset [n]$ and $T_v \in \mathscr{C}(H_v=H)$ for all $v\in L$. Then by Proposition \ref{Prop: cutset of corona}(1), $T_0\neq \emptyset$. Let $N:=\{v\in L : T_v\neq \emptyset\}$. Then using Proposition \ref{Prop: cutset of corona}(2), it follows that $N\subset T_0\cap L$. To prove this converse part, it is sufficient to check that $\omega(G \setminus T)=|T|+1$.
    
{\bf Claim.} $T_0\subset L$. 

{\it Proof of Claim.} --- If not, let $w \in T_0\setminus L$. Notice that $T_0\setminus \{w\}$ being a proper subset of $[n]$, the induced subgraph $K_n\setminus (T_0\setminus \{w\})$ is non-empty and connected. Thus using \ref{Prop: cutset of corona}(5) we get that

\begin{align*}
    \omega(G \setminus (T\setminus \{w\})&=\omega(K_n\setminus (T_0\setminus \{w\})+\sum\limits_{v\in N}\omega(H_v\setminus T_v)+|(T_0\setminus \{w\})\cap L|-|N|,\\
    &=1+\sum\limits_{v\in N}\omega(H_v\setminus T_v)+|T_0\cap L|-|N|,\quad (\text{as } w\notin L)\\
    &\geq\omega(G\setminus T_0)+\sum\limits_{v\in N}\omega(H_v\setminus T_v)+|T_0\cap L|-|N|,\quad (\text{equality holds iff } T_0\subsetneq [n])\\
    &=\omega(G\setminus T),
\end{align*}
contradicting the fact that $T\in \mathscr{C}(G)$. This proves the claim.

Once again it follows from Proposition \ref{Prop: cutset of corona}(5) that 
\begin{align*}
\omega(G \setminus T)&=\omega(K_n\setminus T_0)+\sum\limits_{v\in N}\omega(H_v\setminus T_v)+|T_0\cap L|-|N|,\\
    &=1+\sum\limits_{v\in N}\omega(H_v\setminus T_v)+|T_0|-|N|, \quad (\text{as } T_0\subset L\subsetneq [n])\\
    &=1+\sum\limits_{v\in N}(|T_v|+1)+|T_0|-|N|,\\
    &\hspace{1cm} (\text{as } \emptyset \neq T_v\in \mathscr{C}(H_v), \,\forall v \in N \text{ by Proposition }\ref{Prop: cutset of corona}(3) \text{ and } H \text{ is unmixed})\\
    &=1+\sum\limits_{v\in N}|T_v|+|T_0|,\\
    &=|T|+1,
\end{align*}
as desired. This completes the proof.
\end{proof}

As we remarked earlier in the Introduction, Kiani and Saeedi Madani proved that for two graphs $G$ and $H$ with $|V(G)|\geq 2$, the unmixedness of $J_{G\circ H}$ is equivalent to the Cohen-Macaulayness of $J_{G\circ H}$ and is further equivalent to the fact that both $G$ and $H$ are complete graphs on respective number of vertices (cf. \cite[Theorem 5.4]{KSM15}). 

Unlike the case of the usual corona products, in the case of almost all other generalized corona products, we observe the following interesting result, which provides abundant examples of unmixed but not Cohen-Macaulay binomial edge ideals in this context.

\begin{cor}\label{Cor: t-corona with complete is CM iff the other is CM}
    Let $H$ be any graph. For integers $n$ and $\ell$ such that $n\geq 2$ and $1\leq \ell<n$, $K_n\circ_\ell H$ is Cohen-Macaulay if and only if $H$ is Cohen-Macaulay.
\end{cor}

\begin{proof} --- The proof follows immediately from Corollary \ref{Cor: CM-def}\footnote{For the proof of Corollary \ref{Cor: CM-def}, we refer the reader to Subsection \ref{Subsection: Depth and regularity}.}.
\end{proof}

\begin{cor}\label{Cor: All L-corona are unmixed implies that the first graph is complete}
    Let $H$ be any graph such that $J_H$ is unmixed. For a graph $G$ with $|V(G)|>1$, $J_{G\circ_L H}$ is unmixed for every non-empty proper subset $L$ of $V(G)$ if and only if $G$ is a complete graph.
\end{cor}
\begin{proof} ---
    The direction of proving unmixedness of all generalized corona product $G\circ_L H$ assuming $G$ is a complete graph straight away follows from Corollary \ref{Cor: t-corona with complete is unmixed iff the other is unmixed}. 

    We will prove the converse now. Let $|V(G)|=n$. If possible, let $G$ be not a complete graph. Therefore there are $v,w\in V(G)$ such that $\{v,w\}\notin E(G)$. Let $L:=V(G)\setminus \{v,w\}$, which is a proper subset of $V(G)$. By the given hypothesis, $J_{G\circ_{L}H}$ is unmixed. Therefore Theorem \ref{Thm: Unmixedness of L-corona of G and H implies H is unmixed} implies that $G\setminus L$ is connected. However, this leads to a contradiction to our assumption that $\{v,w\}\notin E(G)$. Therefore it follows that any two vertices of $G$ are connected by an edge in $G$, which implies that $G=K_n$, a complete graph with $n$-vertices. This completes the proof.
\end{proof}

\begin{theorem}\label{Thm: Acc. system of L-corona of G and H implies H has acc. system}
    Let $G$ and $H$ be two graphs such that $|V(G)|>1$ and $L$ be a non-empty subset of $V(G)$. If $G\circ_L H$ satisfies an accessible system of cutsets, then $H$ satisfies the same. 
\end{theorem}
\begin{proof} ---
    Suppose $G\circ_L H$ satisfies an accessible system of cutsets. Let $T_H \in \mathscr{C}(H)$ which is non-empty. Choose $v\in L$ and consider $T_H$ as a cutset of $H_v=H$. Define $T:=\{v\}\cup T_H$. Then $T\in \mathscr{C}(G\circ_L H)$ (see the claim in Theorem \ref{Thm: Unmixedness of L-corona of G and H implies H is unmixed}. Since $(G\circ_L H)\setminus T_H$ is connected and $G\circ_L H$ satisfies the accessible system of cutsets, it turns out that there is $u\in T_H$ such that $T\setminus \{u\}\in \mathscr{C}(G\circ_L H)$. To complete the proof, it is enough to show that $T_H\setminus \{u\}\in \mathscr{C}(H_v=H)$.
    
    Let $u'\in T_H\setminus \{u\}$. Thus using the hypothesis, we have the following:
    \begin{align*}
        &\omega((G\circ_LH)\setminus (T\setminus\{u,u'\})) < \omega((G\circ_LH)\setminus (T\setminus\{u\})) \quad (\text{since } T\setminus \{u\}\in \mathscr{C}(G\circ_L H))\\
        \text{or,}\quad &\omega(G\setminus v)+\omega(H_v\setminus (T_H\setminus\{u,u'\})) < \omega(G\setminus v)+\omega(H_v\setminus (T_H\setminus\{u\}));\\
        \text{or,}\quad &\omega\left(H_v\setminus \left((T_H\setminus\{u\})\setminus \{u'\}\right)\right) < \omega(H_v\setminus (T_H\setminus\{u\})).
    \end{align*}
    This completes the proof.
\end{proof}

Now, we prove the converse of the above result in an important special case.

\begin{theorem}\label{Thm: If H has acc. sys of cutsets then l-corona of H with complete graph also has the same}
    Let $H$ be any graph that satisfies an accessible system of cutsets. Then the graph $G:=K_n\circ_\ell H$ for integers $n$ and $\ell$ such that $n\geq 2$ and $1\leq \ell<n$ also satisfies an accessible system of cutsets.
\end{theorem}
\begin{proof} ---
    Assume all the notations as in Corollary \ref{Cor: t-corona with complete is unmixed iff the other is unmixed}. Clearly, $N\neq \emptyset$; otherwise $T=\emptyset$. Let $w\in N$. Therefore $\emptyset \neq T_0 \subsetneq [n]$ and $\emptyset \neq T_w \in \mathscr{C}(H_w)$ using Proposition \ref{Prop: cutset of corona}. Since $H_{w}$ (same as $H$) satisfies an accessible system of cutsets, thus it follows that there exists $u \in T_{w}$ such that $T_w\setminus \{u\} \in \mathscr{C}(H_w)$. Note that $T^*:=T\setminus \{u\} = T_0\cup (T_w\setminus \{u\})\cup(\bigcup_{v \in N\setminus \{w\}}T_v)$.
    
    \indent\indent{\bf Claim.} {\it $T^* \in \mathscr{C}(G)$.}
    
    \indent\indent{\it Proof of Claim.} --- Let $u' \in T^*$. Then $u'$ lies either in $T_0$, or in $T_w$, or in $T_{v'}$ for some $v'\in N\setminus \{w\}$.
    
    \indent\indent{\bf Case 1.} {\it Assume that $u' \in T_0$.}
    
    In this case, $T^*\setminus \{u'\} = (T_0\setminus \{u'\})\cup (T_w\setminus \{u\})\cup(\bigcup_{v \in N\setminus \{w\}}T_v)$. Since $T \in \mathscr{C}(G)$ thus using Proposition \ref{Prop: cutset of corona}(5), it follows that
    \begin{align*}
        &\omega(G\setminus (T\setminus\{u'\})) < \omega(G\setminus T);\\
        \text{or,}\quad &\omega(K_n\setminus (T_0\setminus \{u'\}))+\sum\limits_{v\in N}\omega(H_v\setminus T_v)+|(T_0\setminus\{u'\})\cap L| -|N|\\
        & < \omega(K_n\setminus T_0)+\sum\limits_{v\in N}\omega(H_v\setminus T_v)+|T_0\cap L| -|N|;\\
        \text{or,}\quad &\omega(K_n\setminus (T_0\setminus \{u'\}))+\sum\limits_{v\in N\setminus \{w\}}\omega(H_v\setminus T_v)+\omega(H_w\setminus(T_w\setminus\{u\}))+|(T_0\setminus\{u'\})\cap L| -|N|\\
        & < \omega(K_n\setminus T_0)+\sum\limits_{v\in N\setminus \{w\}}\omega(H_v\setminus T_v)+\omega(H_w\setminus(T_w\setminus\{u\}))+|T_0\cap L| -|N|;\\
        \text{or,}\quad &\omega(G\setminus (T^*\setminus \{u'\})) < \omega(G\setminus T^*).
    \end{align*}

    \indent\indent{\bf Case 2.} {\it Assume that $u' \in T_{v'}$ for some $v'\in N\setminus \{w\}$.}
    
    In this case, $T^*\setminus \{u'\} = T_0\cup (T_w\setminus \{u\})\cup(T_{v'}\setminus \{u'\}))\cup(\bigcup_{v \in N\setminus \{w,v'\}}T_v)$. Exactly as above, since $T \in \mathscr{C}(G)$, Proposition \ref{Prop: cutset of corona}(5) implies the following: 

    \begin{align*}
        &\omega(G\setminus (T\setminus\{u'\})) < \omega(G\setminus T);\\
        \text{or,}\quad &\omega(K_n\setminus T_0)+\sum\limits_{v\in N\setminus \{v'\}}\omega(H_v\setminus T_v)+\omega(H_{v'}\setminus (T_{v'}\setminus\{u'\})+|T_0\cap L| -|N|\\
        & < \omega(K_n\setminus T_0)+\sum\limits_{v\in N}\omega(H_v\setminus T_v)+|T_0\cap L| -|N|;\\
        \text{or,}\quad &\omega(K_n\setminus T_0)+\sum\limits_{v\in N\setminus \{w,v'\}}\omega(H_v\setminus T_v)+\omega(H_{v'}\setminus (T_{v'}\setminus\{u'\})+\omega(H_w\setminus(T_w\setminus\{u\}))+|T_0\cap L| -|N|\\
        & < \omega(K_n\setminus T_0)+\sum\limits_{v\in N\setminus\{w\}}\omega(H_v\setminus T_v)+\omega(H_w\setminus(T_w\setminus\{u\}))+|T_0\cap L| -|N|;\\
        \text{or,}\quad &\omega(G\setminus \left(T^*\setminus \{u'\}\right)) < \omega(G\setminus T^*).
    \end{align*}

    \indent\indent{\bf Case 3.} {\it Assume that $u' \in T_w$.}
    
    Observe that we haven't yet used the fact that $u\in T_w$ was so chosen that $T_w\setminus \{u\}$ becomes a cutset of $H_w$. But in this case, we use this fact as follows. In this case, $T^*\setminus \{u'\} = T_0\cup ((T_w\setminus \{u\})\setminus \{u'\})\cup(\bigcup_{v \in N\setminus \{w\}}T_v)$. Therefore exactly as above two cases, since $T \in \mathscr{C}(G)$, Proposition \ref{Prop: cutset of corona}(5) implies that 

    \begin{align*}
        &\omega(H_w\setminus\left((T_w\setminus\{u\})\setminus \{u'\}\right))<\omega(H_w\setminus(T_w\setminus\{u\}));\\
        \text{or,}\quad &\omega(K_n\setminus T_0)+\sum\limits_{v\in N\setminus \{w\}}\omega(H_v\setminus T_v)+\omega(H_w\setminus\left((T_w\setminus\{u\})\setminus \{u'\}\right))+|T_0\cap L| -|N|\\
        & < \omega(K_n\setminus T_0)+\sum\limits_{v\in N\setminus\{w\}}\omega(H_v\setminus T_v)+\omega(H_w\setminus(T_w\setminus\{u\}))+|T_0\cap L| -|N|;\\
        \text{or,}\quad &\omega(G\setminus (T^*\setminus \{u'\})) < \omega(G\setminus T^*).
    \end{align*}

Therefore, Cases 1--3 together prove our claim.

Hence, from all the above cases, we conclude that for every $\emptyset \neq T\in \mathscr{C}(G)$, there exists an element $u \in T$ such that $T\setminus \{u\}\in \mathscr{C}(G)$. This proves that $G$ satisfies an accessible system of cutsets, as desired.
\end{proof}
\begin{cor}\label{Cor: t-corona with complete is acc. iff the other is acc.}
    Let $H$ be any graph. For integers $n$ and $\ell$ such that $n\geq 2$ and $1\leq \ell<n$, $K_n\circ_\ell H$ is accessible if and only if $H$ is accessible.
\end{cor}
\begin{proof} ---
    The proof is straightforward by using Corollary \ref{Cor: t-corona with complete is unmixed iff the other is unmixed} followed by Theorem \ref{Thm: If H has acc. sys of cutsets then l-corona of H with complete graph also has the same}.
\end{proof}

Now the above corollaries \ref{Cor: t-corona with complete is unmixed iff the other is unmixed}, \ref{Cor: t-corona with complete is CM iff the other is CM} and \ref{Cor: t-corona with complete is acc. iff the other is acc.} together yield the following interesting result. 

\begin{theorem}\label{Thm: Combined -- l-corona with complete is unmixed, accessible, CM iff the other is unmixed, accessible, CM respectively}
    Let $H$ be any graph. For integers $n$ and $\ell$ such that $n\geq 2$ and $1\leq \ell<n$, $K_n\circ_\ell H$ is unmixed (or accessible or, Cohen-Macaulay) if and only if $H$ is unmixed (or, accessible or, Cohen-Macaulay respectively).
\end{theorem}

\subsection{A small remark on a conjecture of Bolognini-Macchia-Strazzanti}\mbox{}

In \cite{BMS22}, the authors made the following conjecture:

\begin{conj}[{Conjecture 1.1, \cite{BMS22}}]\label{Conj: BMS}
Let $G$ be a graph. Then, $J_G$ is Cohen-Macaulay if and only if $G$ is accessible.
\end{conj}

The combinatorial data of a graph captures a lot of information about the algebraic properties of the associated binomial edge ideal. For instance, we already know the unmixedness of a binomial edge ideal associated with a graph, which is completely understood by counting the size of every cutset and the number of connected components of the induced subgraph by deleting the cutset. So, the unmixedness of a binomial edge ideal is a combinatorial property, too. With a similar spirit, the main motivation behind the above conjecture is to investigate whether the Cohen-Macaulayness of a binomial edge ideal associated with a graph can also be described purely combinatorially or not.

In \cite{BMS22}, the authors proved that if $J_G$ is Cohen-Macaulay, then $G$ is accessible. But the other direction is open, in general. In the same article, the authors settled the above conjecture if $G$ is either a chordal or a traceable graph.

Let $\mathcal{G}$ denote the class of all connected simple graphs. Define the BMS locus as follows:
$${\rm BMS}(\mathcal{G}):=\{G \text{ satisfies B-M-S Conjecture } \ref{Conj: BMS}\}.$$

So the Conjecture \ref{Conj: BMS} can be reformulated as --- ``\emph{Is ${\rm BMS}(\mathcal{G})=\mathcal{G}$?}''

With respect to the standard distance metric $d$ in a graph $G$, the diameter of $G$ is denoted by $\diam(G)$ and is defined as
$$\diam(G):=\max\{d(u,v):u,v\in V(G)\}.$$
Define for any positive integer $k$,
$$\mathcal{D}_k(\mathcal{G}):=\{G\in \mathcal{G} : \diam(G)\leq k\},$$
and further define for $k\geq 2$, $$\mathcal{D}_k^\circ(\mathcal{G}):=\mathcal{D}_k(\mathcal{G})\setminus\mathcal{D}_{k-1}(\mathcal{G}).$$

Then there is a filtration:
$$\mathcal{D}_1(\mathcal{G})\subsetneq \mathcal{D}_2(\mathcal{G})\subsetneq \cdots,$$
such that $\mathcal{G}=\bigcup_{k\geq 1}\mathcal{D}_k(\mathcal{G})=\mathcal{D}_1(\mathcal{G})\cup\left(\bigcup_{k\geq 2}\mathcal{D}_k^\circ(\mathcal{G})\right)$.

We will prove here the following:

\begin{theorem}
    The following assertions hold:
    \begin{enumerate}
        \item If $\mathcal{D}_2^\circ(\mathcal{G})\subseteq {\rm BMS}(\mathcal{G})$, then ${\rm BMS}(\mathcal{G})=\mathcal{G}$.
        \item If $\mathcal{D}_3^\circ(\mathcal{G})\subseteq {\rm BMS}(\mathcal{G})$, then ${\rm BMS}(\mathcal{G})=\mathcal{G}$.
    \end{enumerate}
\end{theorem}
\begin{proof} ---
Let $H$ be a (connected) accessible graph.

    {\it Proof of (1).} --- Consider $\wt{H}:=\{w\}\sqcup H$, consisting of exactly two connected components---an isolated vertex $w$ and $H$. Define $G:={\rm cone}(v, \wt{H})$ for an external vertex $v$. Using \cite[Theorem 4.8(3)]{BMS22}, it follows that $G$ is accessible. For any two distinct vertices $v_1$ and $v_2$ of $G$, observe that
    \[
    {\rm dist}_G(v_1,v_2)=
    \begin{cases}
        1, \quad \text{if } v=v_i \text{ for some } i\in \{1,2\};\\
        1, \quad \text{if } (v_1, v_2) \in V(H)\times V(H) \text{ with } \{v_1,v_2\}\in E(H);\\
        2, \quad \text{otherwise},
    \end{cases}
    \] which implies that $\diam(G)=2$. Thus by the hypothesis, $G \in {\rm BMS}(\mathcal{G})$. Therefore, $J_G$ is Cohen-Macaulay. Using \cite[Theorem 4.8(4)]{BMS22}, it follows that $J_H$ is Cohen-Macaulay, as desired. 
    
    {\it Proof of (2).} --- Consider a (connected) complete graph $K_3$ with three vertices namely $w_1, w_2, w_3$. Define $G:=K_3\circ_{\{w_1,w_2\}}H$. Using Theorem \ref{Thm: Combined -- l-corona with complete is unmixed, accessible, CM iff the other is unmixed, accessible, CM respectively}, it follows that $G$ is accessible. For any two distinct vertices $v_1$ and $v_2$ of $G$, observe that
    \[
    {\rm dist}_G(v_1,v_2)=
    \begin{cases}
        3, \quad \text{if } v_1, v_2 \in V(H_{w_1})\cup V(H_{w_2}) \text{ such that}\\ \qquad (v_1, v_2) \notin (V(H_{w_1})\times V(H_{w_1}))\cup (V(H_{w_2})\times V(H_{w_2}));\\
        2, \quad \text{if } (v_1, v_2) \in V(H_{w_1})\times V(H_{w_1}) \text{ with } \{v_1,v_2\}\notin E(H_{w_1}), \text{ or}\\ \qquad (v_1, v_2) \in V(H_{w_2})\times V(H_{w_2}) \text{ with } \{v_1,v_2\}\notin E(H_{w_2});\\
        2, \quad \text{if } v_1=w_3 \text{ with } v_2 \in V(H_{w_1})\cup V(H_{w_2}), \text{ or}\\
        \qquad v_2=w_3 \text{ with } v_1 \in V(H_{w_1})\cup V(H_{w_2});\\
        1, \quad \text{otherwise},
    \end{cases}
    \] which implies that $\diam(G)=3$. Thus by the hypothesis, $G \in {\rm BMS}(\mathcal{G})$. Therefore, $J_G$ is Cohen-Macaulay. Once again, using Theorem \ref{Thm: Combined -- l-corona with complete is unmixed, accessible, CM iff the other is unmixed, accessible, CM respectively}, it follows that $J_H$ is Cohen-Macaulay, as desired.

    This completes the proof.
\end{proof}
Therefore, related to Conjecture \ref{Conj: BMS}, we are led to ask the following question:
\begin{question}\label{Ques: Reduction of B-M-S Conjecture}
    For a simple graph $G$ with $\diam(G) = 3$, if the corresponding binomial edge ideal $J_G$ is Cohen-Macaulay, is $G$ accessible?
\end{question}

\section{\bf Dimension, Depth and Regularity of Certain Corona Products}

In this section, we study the dimension, depth and Castelnuovo-Mumford regularity of the binomial edge ideal of the corona product of $G$ and $H$ where $G$ is either a complete graph or a Cohen-Macaulay closed graph (e.g. a path graph). We also prove the existence of extremal Betti numbers.

\subsection{Dimension of corona with complete graph}\label{Subsection: Dimension}\mbox{}
We will compute the dimension of $G=K_n\circ _\ell H$ for any connected graph $H$ and $1\leq \ell\leq n$. In the case of $\ell=n$, the graph $K_n\circ _\ell H$ coincides with the usual corona product $K_n\circ H$.
\begin{theorem}\label{Thm: Dimension of l-corona with complete graph}
    Let $H$ be any connected graph and $n\geq 1$ any integer. Consider the graph $G:=K_n\circ_\ell H$ with $1\leq \ell \leq n$. Then 
    \begin{equation}\label{Eqn: dimension formula}
        \dim(S/J_G)=n-\ell +1+\ell \dim (S_H/J_H),
    \end{equation}
    where $|V(H)|= h$.
\end{theorem}
\begin{proof} ---       
    We use induction on $\ell$ to prove this result. For $\ell=1$, let $v$ be the vertex of $K_n$ along which the copy of $H$ is attached to $K_n$ in $G$. Therefore, $$G=K_n\circ_1H=K_n\cup_{v}(v\ast H)=v\ast(K_{n-1}\sqcup H)=v\ast H',$$ where $H'=K_{n-1}\sqcup H.$ As $H'$ is disconnected, $\dim(S_{H'}/J_{H'})\geq |V(H')|+2$ and hence, it follows from \cite[Theorem 4.4]{KS19} that $$\dim(S/J_G)=\dim(S_{H'}/J_{H'})=n+\dim(S_H/J_H),$$ as desired. This establishes the base case of this induction process.

    Choose any $\ell$ with $1\leq \ell \leq n$. Now assume the formula in \eqref{Eqn: dimension formula} for $K_n\circ_t H$ for all $t$ with $1\leq t < \ell$. Consider the graph $G=K_n\circ_\ell H$. Let $v\in V(K_n)$ be any vertex such that a copy of $H$ is attached with $v$ in $G$. Then $$G_v=K_{n+h}\circ_{\ell-1}H; \quad G_v\setminus v = K_{n+h-1}\circ_{\ell-1}H; \quad G\setminus v= H \sqcup (K_{n-1}\circ_{\ell-1}H).$$ Therefore using the induction hypothesis, we get the following.
    \begin{equation}\label{eqn: Dimension components following Ohtani's SES}
\left.
\begin{aligned}
    \dim(S/J_{G_v}) &= n+h-(\ell-1)+1+(\ell-1)\dim (S_H/J_H);\\
    \dim(S/((x_v,y_v)+J_{G_v\setminus v})) &= n+h-1-(\ell-1)+1+(\ell-1)\dim (S_H/J_H);\\
    \dim(S/((x_v,y_v)+J_{G\setminus v})) &= \dim (S_H/J_H) + (n-1)-(\ell-1)+1+(\ell-1)\dim (S_H/J_H)\\
                                         &= \dim(S/J_{G_v})+\dim (S_H/J_H)-(h+1)\\&\geq \dim(S/J_{G_v}), \quad  \text{ as }\dim(S_H/J_H)\geq h+1.
\end{aligned}
\quad\right\}
\end{equation}

Therefore, from the short exact sequence \eqref{ohtani-ses}, we have
\begin{equation}\label{eqn: dimension for SES}
    \max\{\dim(S/((x_v,y_v)+J_{G\setminus v})), \dim(S/J_{G_v})\}=\max\{\dim (S/J_G), \dim(S/((x_v,y_v)+J_{G_v\setminus v}))\}.
\end{equation}
From \eqref{eqn: Dimension components following Ohtani's SES}, it is clear that $\dim(S/((x_v,y_v)+J_{G_v\setminus v}))=\dim(S/J_{G_v})-1$ and thus it follows from \eqref{eqn: dimension for SES} that 
\begin{align*}
    \dim(S/J_G)&=\dim(S/((x_v,y_v)+J_{G\setminus v});\\
               &= n-\ell+1+\ell\dim (S_H/J_H),
\end{align*}
as desired. This completes the proof.
\end{proof}

\begin{cor}\label{Cor: dimension of complete corona}
    For any connected graph $H$ and any integer $n\geq 1$, 
    \begin{equation}\label{eqn: dimension of complete corona}
        \dim(S/J_{K_n\circ H})=n\dim(S_H/J_H)+1,
    \end{equation}
    where $|V(H)|=h$.
\end{cor}
\begin{proof} ---
    Immediate from Theorem \ref{Thm: Dimension of l-corona with complete graph}.
\end{proof}

\begin{cor}\label{Cor: dimension of complete corona with unmixed H}
    Let $H$ be any connected graph such that $J_H$ is unmixed. Then for any integer $n\geq 1$, 
    \begin{equation}\label{eqn: dimension of complete corona with unmixed H}
        \dim(S/J_{K_n\circ H})=n+nh+1,
    \end{equation}
    where $|V(H)|=h$.
\end{cor}
\begin{proof} ---
    Since $J_H$ is unmixed, $\dim(S_H/J_H)=h+1$. Plug this into \eqref{eqn: dimension of complete corona} to obtain the desired result.
\end{proof}
\begin{remark}
    The result of Corollary \ref{Cor: dimension of complete corona with unmixed H} was first observed in \cite[Corollary 5.2]{KSM15}. 
\end{remark}
\subsection{Depth and regularity of corona with complete graph}\label{Subsection: Depth and regularity}\mbox{}

    Our next aim is to compute the depth and regularity of $G=K_n\circ _\ell H$ for any connected graph $H$ and $1\leq \ell\leq n$. As noted earlier, in the case of $\ell=n$, the graph $K_n\circ _\ell H$ coincides with the usual corona product $K_n\circ H$. But we will see that the depth and regularity formulae for $K_n\circ _\ell H$ with $1\leq \ell< n$ are slightly different from those with the case of $\ell=n$. 
    
    We first consider the base case when $\ell=1$. If $n=1$, then $G=v* H$, and hence, by virtue of \cite[Lemma 3.4]{KS20}, $\depth(S/J_G)=\depth(S_H/J_H)$. So we assume that $n\geq 2$.

\begin{prop}\label{Prop: Depth for 1-corona with complete graph}
    Let $H$ be a connected graph with $|V(H)|=h$ and $n\geq 2$ be any integer. Consider the graph $G:=K_n\circ_1 H$. Then the following assertions hold.
    \begin{enumerate}
        \item $\depth(S/J_G)=n+\depth(S_H/J_H)$.
        \item $\reg(S/J_G)=
        \begin{cases}
            2, \quad &\text{if } H \text{ is complete};\\
            1+\reg(S_H/J_H), &\text{otherwise}.
        \end{cases}
        $
        \item Assume that $H$ is not a complete graph, $\pd(S_H/J_H)=p_H$ and $\beta_{p_H,p_H+r_H}(S_H/J_H)$ is an extremal Betti number of $S_H/J_H$ for $r_H\geq 2$. Then $\beta_{p,p+r_H}(S/J_G)$ is an extremal Betti number if $n=2$ and $\beta_{p,p+r_H+1}(S/J_G)$ is an extremal Betti number if $n\geq 3$, where $p=\pd(S/J_G)=n+p_H.$
    \end{enumerate}
\end{prop}









\begin{proof} ---
We will prove the first two statements together.

    {\it Proof of (1) and (2).} ---
    Assume that $|V(H)|=h$. Let $S_{H}$ be the polynomial ring over $\K$ associated with the graph $H$.

    {\bf Case 1.} {\it Whenever $H=K_{h}$.}

    In this case, it follows from \cite[Theorem 1.1]{EHH11} that $G$ is a block graph with $\depth(G)=|V(G)|+1=n+h+1=n+\depth(S_H/J_H)$, as desired.

    In this case, by \cite[Theorem 8]{HR18}, $\reg(S/J_G)=\iv(G)+1=2$.

    {\bf Case 2.} {\it Whenever $H$ is not complete.}

     Let $v$ be the vertex of $K_n$ along which $H$ is attached to $K_n$ in $G$. 
     Therefore, $$G=K_n\cup_{v}(v\ast H)=v\ast(K_{n-1}\sqcup H).$$
     
     From \cite[Theorem 3.9]{KS20}, it follows that $\depth(S/J_G) = \min\{\depth(K_{n-1}\sqcup H), n+h+1\}$.
     
     Clearly, $\depth(K_{n-1}\sqcup H)=n+\depth(S_H/J_H) \leq n+h+1
     $, as $\depth(S_H/J_H)\leq h+1$.
     Hence, whenever $n\geq 2$, $$\depth(S/J_G)=\min\{n+\depth(S_H/J_H), n+h+1\}=n+\depth(S_H/J_H).$$
    Again using \cite[Theorem 2.1(a)]{SMK18}, it immediately follows that $$\reg(S/J_G)=\max\{1+\reg(S_H/J_H),2\}=1+\reg(S_H/J_H),$$ since $H$ being non-complete connected graph, $\reg(S_H/J_{H})>1$.

     {\it Proof of (3).} --- Assume that $\pd(S_H/J_H)=p_H$. Then by \cite[Lemma 3.3]{KS20}, there exists an integer $r_H\geq 2$ such that $\beta_{p_H,p_H+r_H}(S_H/J_H)$ is an extremal Betti number of $S_H/J_H$. Also, by part (1), we get $\depth(S/J_G)=n+\depth(S_H/J_H)$ and hence, Auslander-Buchsbaum formula, we have $p:=\pd(S/J_G)=n+p_H$. Since $v$ is a non-simplicial vertex of $G$, we consider the short exact sequence \ref{ohtani-ses}.
Observe that, $$G_v=K_{n+h}; \quad G_v\setminus v = K_{n+h-1}; \quad G\setminus v= K_{n-1} \sqcup H.$$ Therefore, $\pd(S/J_{G_v})=n+h-1$, $\pd(S/((x_v,y_v)+J_{G_v\setminus v}))=n+h$ and $\pd(S/((x_v,y_v)+J_{G\setminus v}))=n+p_H=p$. Hence, $\beta_{n+h-1,n+h}(S/J_{G_v})$ and $\beta_{n+h,n+h+1}(S/((x_v,y_v)+J_{G_v\setminus v}))$ are respective extremal Betti numbers. If $n=2$, then $\beta_{p,p+r_H}(S/((x_v,y_v)+J_{G\setminus v}))$ and if $n\geq 3$, then $\beta_{p,p+r_H+1}(S/((x_v,y_v)+J_{G\setminus v}))$ is an extremal Betti number.
By \cite[Theorem B]{BN17}, $p_H\geq h-1$. If $p_H=h-1$, then considering the long exact sequence of $\Tor$ \eqref{ohtani-tor} in homological degree $p$, we have
\begin{equation*}
    0\to \Tor^{S}_{p+1}\left(\frac{S}{(x_v,y_v)+J_{G_v\setminus v}},\mathbf{k}\right)_{p+j} \to \Tor_{p}^{S}\left( \frac{S}{J_G},\K\right)_{p+j} \to \Tor_{p}^{S}\left( \frac{S}{(x_v,y_v)+J_{G\setminus v}},\K\right)_{p+j}\to \cdots
\end{equation*}
Also if $p_H\geq h$, then we have the following isomorphism:
\begin{equation*}
    \Tor^{S}_{p}\left( \frac{S}{J_G},\mathbf{k}\right)_{p+j} \simeq
    \Tor^{S}_{p}\left(\frac{S}{(x_v,y_v)+J_{G\setminus v}},\mathbf{k}\right)_{p+j} \text{ for }j\geq 2.
\end{equation*}
This implies that if $n=2$, then $\beta_{p,p+r_H}(S/J_G)\neq 0$ and $\beta_{p,p+j}(S/J_G)=0$ for all $j\geq r_H+1$. Also, if $n\geq 3$, then $\beta_{p,p+r_H+1}(S/J_G)\neq 0$ and $\beta_{p,p+j}(S/J_G)=0$ for all $j\geq r_H+2$. Hence, either $\beta_{p,p+r_H}(S/J_G)$ or $\beta_{p,p+r_H+1}(S/J_G)$ is an extremal Betti  number of $S/J_G$.
\end{proof}

\begin{theorem}\label{Thm: Depth and regularity for t-corona with complete graph}
    For any connected graph $H$ and any integer $n\geq 3$, consider the graphs $G:=K_n\circ_\ell H$ for some $1\leq \ell < n$. Then the following assertions hold.
    \begin{enumerate}
        \item $\depth(S/J_G)=n-\ell+1+\ell\cdot\depth(S_H/J_H)$.
        \item $\reg(S/J_G)=1+\ell\cdot\reg(S/J_H)$.
        \item Assume that $H$ is not a complete graph, $\pd(S_H/J_H)=p_H$ and $\beta_{p_H,p_H+r_H}(S_H/J_H)$ is an extremal Betti number of $S_H/J_H$ for $r_H\geq 2$. Then $\beta_{p,p+\ell r_H}(S/J_G)$ is an extremal Betti number if $n=2$ and $\beta_{p,p+\ell r_H+1}(S/J_G)$ is an extremal Betti number if $n\geq 3$, where $p=\pd(S/J_G)=n+\ell-1+\ell p_H.$
    \end{enumerate}
\end{theorem}
\begin{proof} ---
Assume that $|V(H)|=h$. 

{\bf Case 1.} {\it Whenever $H=K_h$.} 

In this case, $G$ is a block graph, and thus it follows from \cite[Theorem 1.1]{EHH11} and \cite[Theorem 8]{HR18} that
$$\depth(S/J_G)=|V(K_n\circ_\ell K_h)|+1=n+\ell h+1=n-\ell+1+\ell\depth(S_{K_h}/J_{K_h})$$ and $\reg(S/J_G)=\iv(G)+1=1+\ell$, respectively.

{\bf Case 2.} {\it Whenever $H$ is non-complete.}

Let $\pd(S_H/J_H)=p_H$. Then by Auslander-Buchsbaum formula, $\depth(S_H/J_H)=2h-p_H$. To prove the assertion we show that $\depth(S/J_G)\geq n-\ell+1+\ell\depth(S_H/J_H)$, $\reg(S/J_G)=1+\ell\reg(S_H/J_H)$. To prove $\depth(S/J_G)\leq n-\ell+1+\ell\depth(S_H/J_H)$, we show that if $n=2$, then $\beta_{p,p+\ell r_H}(S/J_G)$ is an extremal Betti number and if $n\geq 3$, then $\beta_{p,p+\ell r_H+1}(S/J_G)$ is an extremal Betti number, where $p:=n+\ell-1+\ell p_H$. We proceed by induction on $\ell$. The base case $\ell=1$ for the induction follows from Proposition \ref{Prop: Depth for 1-corona with complete graph}. Assume that $\ell\geq 2$. 


 Let $v$ be one of the vertices along which $H$ is attached to $K_n$ in $G$. 
Observe that, $$G_v=K_{n+h}\circ_{(\ell-1)}H; \quad G_v\setminus v = K_{n+h-1}\circ_{(\ell-1)}H; \quad G\setminus v= H \sqcup (K_{n-1}\circ_{(\ell-1)}H).$$
Using the induction hypothesis, we get the following.
\begin{equation}
\left.
\begin{aligned}
    \depth(S/J_{G_v}) &= n+h-\ell+2+(\ell-1)\depth(S_H/J_H);\\
    \depth(S/((x_v,y_v)+J_{G_v\setminus v})) &= n+h-\ell+1+(\ell-1)\depth(S_H/J_H);\\
    \depth(S/((x_v,y_v)+J_{G\setminus v})) &= n-\ell+1+\ell\depth(S_H/J_H);
\end{aligned}
\quad\right\}
\end{equation}
and
\begin{equation}
\left.
\begin{aligned}
    \reg(S/J_{G_v}) &= 1+(\ell-1)\reg(S_H/J_H);\\
    \reg(S/((x_v,y_v)+J_{G_v\setminus v})) &=1+(\ell-1)\reg(S_H/J_H);\\
    \reg(S/((x_v,y_v)+J_{G\setminus v})) &=1+\ell\reg(S_H/J_H).
\end{aligned}
\quad\right\}
\end{equation}
By Auslander-Buchsbaum formula, we have 
\begin{align*}
    \pd(S/J_{G_v}) & =n+h-(\ell-1)-1+(\ell-1)p_H=p+h-1-p_H;\\
    \pd(S/((x_v,y_v)+J_{G_v\setminus v})) &=p+h-2-p_H;\\
    \pd(S/((x_v,y_v)+J_{G\setminus v})) &=p_H+n-1+\ell-1-1+(\ell-1)p_H=p-2.
\end{align*}
Therefore
\begin{align*}
    \depth(S/J_{G_v}\oplus S/((x_v,y_v)+J_{G\setminus v}))&=\min\{\depth(S/J_{G_v}), \depth(S/((x_v,y_v)+J_{G\setminus v}))\}\\
    &=n-\ell+1+\ell\depth(S_H/J_H), \quad (\text{since } \depth(H) \leq h+1);\\
    \depth(S/((x_v,y_v)+J_{G_v\setminus v}))&=n+h-\ell+1+(\ell-1)\depth(S_H/J_H);
\end{align*}
and
\begin{align*}
    \reg(S/J_{G_v}\oplus S/((x_v,y_v)+J_{G\setminus v}))&=\max\{\reg(S/J_{G_v}), \reg(S/((x_v,y_v)+J_{G\setminus v}))\}\\
    &=1+\ell\reg(S_H/J_H);\\
    \reg(S/((x_v,y_v)+J_{G_v\setminus v}))&=1+(\ell-1)\reg(S_H/J_H).
\end{align*}
Using Depth Lemma(a) and Regularity Lemma (c) on the short exact sequence \eqref{ohtani-ses}, it immediately follows that $$\depth(S/J_G)\geq n-\ell+1+\ell\depth(S_H/J_H) \text{ and } \reg(S/J_G)=1+\ell\reg(S_H/J_H).$$
Since $n+h-1\geq 3$, by induction hypothesis, 
\[
\beta_{p+h-1-p_H,p+h-1-p_H+(\ell-1)r_H+1}(S/J_{G_v})
\text{ and }
\beta_{p+h-p_H,p+h-p_H+(\ell-1)r_H+1}(S/((x_v,y_v)+J_{G_v\setminus v}))
\]
are respective extremal Betti numbers. Also, if $n=2$, then $\beta_{p,p+\ell r_H}(S/((x_v,y_v)+J_{G\setminus v}))$ and if $n\geq 3$, then $\beta_{p,p+\ell r_H+1}(S/((x_v,y_v)+J_{G_v\setminus v}))$ are extremal Betti numbers. 
By \cite[Theorem B]{BN17}, $p_H\geq h-1$. If $p_H=h-1$, then considering the long exact sequence of $\Tor$ \eqref{ohtani-tor} in homological degree $p$ and for $j\geq (\ell-1)r_H+2$, we have 
\begin{equation*}
    0\to \Tor^{S}_{p+1}\left(\frac{S}{(x_v,y_v)+J_{G_v\setminus v}},\mathbf{k}\right)_{p+j} \to \Tor_{p}^{S}\left( \frac{S}{J_G},\K\right)_{p+j} \to \Tor_{p}^{S}\left( \frac{S}{(x_v,y_v)+J_{G\setminus v}},\K\right)_{p+j}\to \cdots
\end{equation*}
Also if $p_H\geq h$, then we have the following isomorphism:
\begin{equation*}
    \Tor^{S}_{p}\left( \frac{S}{J_G},\mathbf{k}\right)_{p+j} \simeq
    \Tor^{S}_{p}\left(\frac{S}{(x_v,y_v)+J_{G\setminus v}},\mathbf{k}\right)_{p+j} \text{ for }j\geq \ell r_H.
\end{equation*}
This implies that if $n=2$, then $\beta_{p,p+\ell r_H}(S/J_G)\neq 0$ and $\beta_{p,p+j}(S/J_G)=0$ for all $j\geq \ell r_H+1$. Also, if $n\geq 3$, then $\beta_{p,p+\ell r_H+1}(S/J_G)\neq 0$ and $\beta_{p,p+j}(S/J_G)=0$ for all $j\geq \ell r_H+2$. Therefore, either $\beta_{p,p+\ell r_H}(S/J_G)$ or $\beta_{p,p+\ell r_H+1}(S/J_G)$ is an extremal Betti  number of $S/J_G$. Hence, the assertion follows.
\end{proof}

\begin{theorem}\label{Thm: Depth for corona with complete graph}
For any connected graph $H$ and any integer $n\geq 1$, let $G:=K_n\circ H$. Then the following assertions hold.
    \begin{enumerate}
        \item $\depth(S/J_G)=
        \begin{cases}
            1+n\cdot\depth(S_H/J_H), & \text{whenever } H \text{ is complete};\\
            n\cdot\depth(S_H/J_H), & \text{whenever } H \text{ is non-complete}.
        \end{cases}
        $\\
        \item $\reg(S/J_G)=
        \begin{cases}
            n+1, & \text{whenever } H \text{ is complete};\\
            n\cdot\reg(S_H/J_H), & \text{whenever } H \text{ is non-complete}.
        \end{cases}
        $
        \item Assume that $H$ is not a complete graph, $\pd(S_H/J_H)=p_H$ and $\beta_{p_H,p_H+r_H}(S_H/J_H)$ is an extremal Betti number of $S_H/J_H$ for $r_H\geq 2$. Then $\beta_{p,p+nr_H}(S/J_G)$ is an extremal Betti number if $n=2$ and $\beta_{p,p+nr_H+1}(S/J_G)$ is an extremal Betti number if $n\geq 3$, where $p=\pd(S/J_G)=2n+np_H.$
    \end{enumerate}
\end{theorem}
\begin{proof} ---
    Assume that $|V(H)|=h$. 

    {\bf Case 1.} {\it Whenever $H=K_h$.} 

    In this case, $G$ is a block graph and it follows from \cite[Theorem 1.1]{EHH11} that
\begin{align*}
    \depth(S/J_G)=|V(K_n\circ K_h)|+1 =n+nh+1=1+n\depth(S_{K_h}/J_{K_h}).
\end{align*}
Also, by \cite[Theorem 8]{HR18}, $\reg(S/J_G)=\iv(G)+1=n+1.$

{\bf Case 2.} {\it Whenever $H$ is non-complete.}

In this case, in order to prove (1), (2) and (3) together, we apply induction on $n$. For the base case $n=1$ of the induction, let $V(K_1):=\{v\}$. Therefore the assertion (1), (2) and (3) follows from \cite[Theorem 3.4]{KS20}, \cite[Theorem 2.1]{SMK18} and \cite[Theorem 3.10]{KS20} respectively. 

Assume that $n\geq 2$. Let $v \in V(K_n)$. Therefore $v$ is a non-simplicial vertex of $G$. Observe that, $$G_v=K_{n+h}\circ_{(n-1)}H; \quad G_v\setminus v = K_{n+h-1}\circ_{(n-1)}H; \quad G\setminus v= H \sqcup (K_{n-1}\circ H).$$
Clearly, $n-1<n+h-1<n+h$ and thus by applying Theorem \ref{Thm: Depth and regularity for t-corona with complete graph} and the induction hypothesis, we get the following.
\begin{equation}\label{eqn7}
\left.
\begin{aligned}
    \depth(S/J_{G_v}) &= h+2+(n-1)\depth(S_H/J_H);\\
    \depth(S/((x_v,y_v)+J_{G_v\setminus v})) &= h+1+(n-1)\depth(S_H/J_H);\\
    \depth(S/((x_v,y_v)+J_{G\setminus v})) &= n\depth(S_H/J_H);
\end{aligned}
\quad\right\}
\end{equation}
and
\begin{equation}\label{eqn8}
\left.
\begin{aligned}
    \reg(S/J_{G_v}) &= 1+(n-1)\reg(S_H/J_H);\\
    \reg(S/((x_v,y_v)+J_{G_v\setminus v})) &=1+(n-1)\reg(S_H/J_H);\\
    \reg(S/((x_v,y_v)+J_{G\setminus v})) &=n\reg(S_H/J_H).
\end{aligned}
\quad\right\}
\end{equation}
Therefore
\begin{align*}
    \depth(S/J_{G_v}\oplus S/((x_v,y_v)+J_{G\setminus v}))&=\min\{\depth(S/J_{G_v}), \depth(S/((x_v,y_v)+J_{G\setminus v}))\}\\
    &=n\depth(S_H/J_H), \quad (\text{since } \depth(S_H/J_H) \leq h+1);
\end{align*}
and
\begin{align*}
    \reg(S/J_{G_v}\oplus S/((x_v,y_v)+J_{G\setminus v}))&=\max\{\reg(S/J_{G_v}), \reg(S/((x_v,y_v)+J_{G\setminus v}))\}\\
    &=n\reg(S_H/J_H), \quad (\text{since } \reg(S_H/J_H) \geq 1).
\end{align*}

Since $H$ is non-complete, $\reg(S_H/J_H)>1$ and thus by using Regularity Lemma (c) on the short exact sequence \eqref{ohtani-ses}, it immediately follows that $\reg(S/J_G)=n\reg(S_H/J_H)$.

Also, by applying Depth Lemma (a) on the short exact sequence \eqref{ohtani-ses}, we have $\depth(S/J_G)\geq n\depth(S_H/J_H)$.

Let $p=2n+np_H$. Now Auslander-Buchsbaum formula yields 
\begin{align*}
    \pd(S/J_{G_v}) =2(n+nh)-(h+2+(n-1)(2h-p_H)) & =p+h-2-p_H<p;\\
    \pd(S/((x_v,y_v)+J_{G_v\setminus v})) &=p+h-1-p_H\leq p;\\
    \pd(S/((x_v,y_v)+J_{G\setminus v})) &=p.
\end{align*}
 It follows from Theorem \ref{Thm: Depth and regularity for t-corona with complete graph} that $\beta_{p+h-1-p_H,p+h-1-p_H+nr_H+1}(S/((x_v,y_v)+J_{G_v\setminus v}))$ is an extremal Betti number. By induction, we have $\beta_{p,p+nr_H}(S/((x_v,y_v)+J_{G\setminus v}))$ is an extremal Betti number if $n=2$ and $\beta_{p,p+nr_H+1}(S/((x_v,y_v)+J_{G\setminus v}))$ is an extremal Betti number if $n\geq 3$.




As $r_H\geq 2$, considering the long exact sequence of $\Tor$ \eqref{ohtani-ses} in homological degree $p$, we have the following isomorphism:
\begin{equation}\label{eqn-tor for corona with complete}
    \Tor^{S}_{p}\left( \frac{S}{J_G},\mathbf{k}\right)_{p+j} \simeq
    \Tor^{S}_{p}\left(\frac{S}{(x_v,y_v)+J_{G\setminus v}},\mathbf{k}\right)_{p+j} \text{ for }j\geq  nr_H.
\end{equation}
Therefore, from the above isomorphism, it turns out that if $n=2$, then $\beta_{p,p+nr_H}(S/J_G)\neq 0$ and $\beta_{p,p+j}(S/J_G)= 0$ for $j\geq nr_H+1$. Also, if $n\geq 3$, then $\beta_{p,p+nr_H}(S/J_G)\neq 0$ and $\beta_{p,p+j}(S/J_G)= 0$ for $j\geq nr_H+2$. Therefore, $\pd(S/J_G)\geq p=2n+np_H$ and thus once again by the Auslander-Buchsbaum formula, $\depth(G)\leq 2|V(G)|-p=n\depth(H).$
This completes the proof.
\end{proof}
\begin{definition}
    For a commutative Noetherian ring $R$ and a finitely generated $R$-module $M$, the \emph{Cohen-Macaulay defect} of $M$, denoted by ${\rm cmdef}(M)$, is defined by $\dim(M)-\depth(M).$ In particular, if ${\rm cmdef}(M)=1$, then $M$ is called \textit{almost Cohen-Macaulay}.
\end{definition}
\begin{cor}\label{Cor: CM-def}
    Let $H$ be a connected graph and $n\geq 1$ an integer. Let $G=K_n\circ_{\ell}H$ for $1\leq \ell \leq n$. Then the following holds:
    \[
    {\rm cmdef}(S/J_G)=
    \begin{cases}
        \ell\cdot {\rm cmdef}(S_H/J_H) & \text{if \ } 1\leq \ell <n;\\
        1+\ell\cdot {\rm cmdef}(S_H/J_H) & \text{if \ } \ell=n \text{ and } H \text{ is not complete}; \\
        0 & \text{if \ } \ell=n \text{ and } H \text{ is complete}.
    \end{cases}
    \]
\end{cor}
\begin{proof}
    The proof is immediate following Theorems \ref{Thm: Dimension of l-corona with complete graph}, \ref{Thm: Depth and regularity for t-corona with complete graph} and \ref{Thm: Depth for corona with complete graph}.
\end{proof}
\begin{remark}
    If $H$ is a non-complete Cohen-Macaulay graph, then it is evident from Corollary \ref{Cor: CM-def} that $K_n\circ H$ is almost Cohen-Macaulay for all $n\geq 1$. Moreover, if $\ell$ is any positive integer, then by considering $G=K_n\circ_{\ell} H'$ for $n>\ell$ and $H'$ almost Cohen-Macaulay, we get that ${\rm cmdef}(S/J_G)=\ell.$ See also \cite[Corollary 3.6]{KS20} for another construction of graphs having arbitrary Cohen-Macaulay defect.
\end{remark}
\subsection{Depth and regularity of corona with Cohen-Macaulay closed graph}\mbox{}

Let $B$ be a connected closed graph such that $J_B$ is Cohen-Macaulay. Then it follows from \cite[Theorem 3.1]{EHH11} that $B$ is a decomposable graph where the indecomposable subgraphs are complete graphs only.
This section considers the graph $G$ obtained by the corona product of Cohen-Macaulay closed graph $B$ and an arbitrary graph $H$. Then we study the depth of their binomial edge ideals.
\begin{theorem}\label{Thm: Depth for corona with CM closed graph}
Let $B$ be a Cohen-Macaulay closed graph and $H$ any connected graph with $|V(B)|=b$. Consider the graph $G:=B\circ H$ for $n \geq 1$. Then 
\begin{enumerate}
    \item $\depth(S/J_G)=
\begin{cases}
1+b \cdot \depth(S_H/J_H), &\text{if } H \text{ is complete};\\
b\cdot \depth(S_H/J_H), &\text{otherwise}.
\end{cases}
$
    \item $\reg(S/J_G)=
        \begin{cases}
            b+1, & \text{ if } H \text{ is complete};\\
            b\cdot\reg(S_H/J_H), & \text{ otherwise }.
        \end{cases}
        $
        \item Assume that $H$ is not a complete graph, $\pd(S_H/J_H)=p_H$ and $\beta_{p_H,p_H+r_H}(S_H/J_H)$ is an extremal Betti number of $S_H/J_H$ for $r_H\geq 2$. Then $\beta_{p,p+br_H+1}(S/J_G)$ is an extremal Betti number, where $p=\pd(S/J_G)=2b+bp_H.$
\end{enumerate}
\end{theorem}
\begin{proof} ---
Without loss of generality, we assume that $b\geq 3$. Otherwise, $B$ is a complete graph and the assertion follows immediately from Theorem \ref{Thm: Depth for corona with complete graph}. Let $|V(H)|=h$.

{\bf Case 1.} {\it Whenever $H=K_h$.} 

Then $G$ is a block graph, and hence, the assertion follows from \cite[Theorem 1.1]{EHH11} and \cite[Theorem 8]{HR18}.

{\bf Case 2.} {\it Whenever $H$ is not complete.}

Since $B$ is a Cohen-Macaulay closed graph,
it follows from \cite[Theorem 3.1]{EHH11} that $B$ is a decomposable graph where the indecomposable subgraphs are complete graphs only. 

Let $B=\bigcup_{i=1}^{k}B_i$ be a decomposition into indecomposable blocks such that each $B_i$ is a complete graph with $|V(B_i)|=b_i$ and 
$$
V(B_i)\cap V(B_j):=
\begin{cases}
    \{v_{ij}\}, &\text{whenever } |i-j|=1;\\
    \emptyset, &\text{otherwise}.
\end{cases}
$$
We prove the result by induction on the number of blocks in $B$. Consider a set of statements $\mathscr{F}(\ell)=(\mathscr{P}(\ell),\mathscr{Q}(\ell))$ as defined below:

\begin{equation}\label{induction 1}
\left.
\begin{aligned}
\mathscr{P}(\ell): &~~\,\depth(S_{B\circ H}/J_{B\circ H})=b\depth(S_H/J_H) \text{ and }\reg(S_{B\circ H}/J_{B\circ H})=b\reg(S_H/J_H), \\ 
&~~ \text{ for a Cohen-Macaulay closed graph } B \text{ with } \ell \text{ number of blocks}.\\[0.15cm] &~~
\text{ Moreover, } \beta_{p,\,p+br_H+1}(S/J_G) \text{ is an extremal Betti number of }S/J_G \text{ where } p=\pd(S/J_G).\\[0.5cm]
\mathscr{Q}(\ell): &~~\text{ Let } X=(K_n\circ_\ell H)\cup_{v}(B\circ H); \text{ for } B \text{ as above}, v \in V(B)\cap V(K_n) \text{ with } B\cup_{v}K_n \text{ is Cohen-}\\
&~~\text{ Macaulay closed graph}, n \geq 3 \text{ any integer and } \ell=|L| \text{ for any subset }L \subsetneq V(K_n)\setminus \{v\}. \text{ Then }\\
&~~\depth(S_X/J_X)=n-\ell+(b+\ell)\depth(S_H/J_H) \text{ and }\reg(S_X/J_X)=1+(b+\ell)\reg(S_H/J_H).\\[0.15cm]
&~~ \text{ Moreover, } \beta_{p,\,p+(b+\ell)p_H}(S_X/J_X) \text{ is an extremal Betti number of } S_X/J_X \text{ where}\\
&~~ \,\,p=2b+n+\ell-2+(b+\ell)p_H.
\end{aligned}
\tag{\textasteriskcentered}
\right\}
\end{equation}

We prove $\mathscr{F}(\ell)$ using induction on $\ell$, that is the total number of blocks in $B$.

In the base case when $\ell=1$, then $B$ turns out to be the complete graph $K_b$. Therefore $\mathscr{P}(1)$ follows immediately from Theorem \ref{Thm: Depth for corona with complete graph}. 

Now we establish $\mathscr{Q}(1)$. Let $X:=(K_n\circ_\ell H)\cup_{v}(K_b\circ H)$. Then 
\begin{align*}
    X_v &= K_{n+b+h-1}\circ_{b+\ell-1}H;\\
    X_v \setminus v &= K_{n+b+h-2}\circ_{b+\ell-1}H; \text{ and}\\
    X\setminus v &=H\sqcup (K_{b-1}\circ H)\sqcup (K_{n-1}\circ_{\ell}H).
\end{align*}
Therefore using Theorem \ref{Thm: Depth and regularity for t-corona with complete graph} and Theorem \ref{Thm: Depth for corona with complete graph} we have 
    \begin{align*}
    \depth(S_X/J_{X_v})&=n-\ell+h+1+(b+\ell-1)\depth(S_H/J_H);\\
    \depth(S_X/((x_v,y_v)+J_{X_v\setminus v}))&=n-\ell+h+(b+\ell-1)\depth(S_H/J_H);\\
    \depth(S_X/((x_v,y_v)+J_{X\setminus v}))&=n-\ell+(b+\ell)\depth(S_H/J_H)
    \end{align*}
    and 
    \begin{align*}
    \reg(S_X/J_{X_v})&=1+(b+\ell-1)\reg(S_H/J_H);\\
    \reg(S_X/((x_v,y_v)+J_{X_v\setminus v}))&=1+(b+\ell-1)\reg(S_H/J_H);\\
    \reg(S_X/((x_v,y_v)+J_{X\setminus v}))&=1+(b+\ell)\reg(S_H/J_H).
    \end{align*}
Now it follows from  regularity lemma that $\reg(X)=1+(b+\ell)\reg(H)$ and by depth lemma, we have $\depth(X)\geq n-\ell+(b+\ell)\depth(H)$. 

Let $p=2b+n+\ell-2+(b+\ell)p_H$. Then by Auslander-Buchsbaum formula, $p_1:=\pd(S_X/J_{X_v})=2b+n+\ell-2+(b+\ell-1)p_H+h-1\leq p$, $p_2:=\pd(S_X/((x_v,y_v)+J_{X_v\setminus v}))=2b+n+\ell-2+(b+\ell-1)p_H+h$ and $\pd(S_X/((x_v,y_v)+J_{X\setminus v}))=p$. Therefore, thanks to Theorems \ref{Thm: Depth and regularity for t-corona with complete graph} and \ref{Thm: Depth for corona with complete graph}, we get that $\beta_{p_1,p_1+(b+\ell-1)r_H+1}(S_X/J_{X_v})$, $\beta_{p_2,p_2+(b+\ell-1)r_H+1}(S_X/((x_v,y_v)+J_{X_v\setminus v})$ and $\beta_{p,p+(b+\ell)r_H+2}(S_X/((x_v,y_v)+J_{X\setminus v})$ are extremal Betti numbers.

By \cite[Theorem B]{BN17}, $p_H\geq h-1$. If $p_H=h-1$, then $p_1=p$ and $p_2=p+1$. Considering the long exact sequence of $\Tor$ \eqref{ohtani-tor} in homological degree $p$ and for $j\geq (b+\ell-1)r_H+2$, we have 
\begin{equation*}
    0\to \Tor^{S_X}_{p+1}\left(\frac{S_X}{(x_v,y_v)+J_{X_v\setminus v}},\mathbf{k}\right)_{p+j} \to \Tor_{p}^{S_X}\left( \frac{S_X}{J_X},\K\right)_{p+j} \to \Tor_{p}^{S_X}\left( \frac{S_X}{(x_v,y_v)+J_{X\setminus v}},\K\right)_{p+j}\to \cdots
\end{equation*}
Also if $p_H\geq h$, then we have the following isomorphism:
\begin{equation*}
    \Tor^{S_X}_{p}\left( \frac{S_X}{J_X},\mathbf{k}\right)_{p+j} \simeq
    \Tor^{S_X}_{p}\left(\frac{S_X}{(x_v,y_v)+J_{X\setminus v}},\mathbf{k}\right)_{p+j} \text{ for }j\geq (b+\ell)r_H.
\end{equation*}
This implies that $\beta_{p,p+(b+\ell)r_H+1}(S/J_G)\neq 0$ and $\beta_{p,p+j}(S/J_G)=0$ for all $j\geq (b+\ell)r_H+2$. Hence, $\mathscr{Q}(1)$ follows.

We now show that $\mathscr{P}(k)$.
Recall that $B=\bigcup_{i=1}^{k}B_i$ is the chosen decomposition of $B$ into its indecomposable blocks with $B_i=K_{b_i}$ for all $1\leq i \leq k$. Clearly,
\begin{equation}\label{eqn10}
    b=|V(B)|=\sum_{i=1}^{k}(b_i-1)+1.
\end{equation}

Now, we first prove $\mathscr{P}(k)$ assuming that $\mathscr{F}(j)=(\mathscr{P}(j), \mathscr{Q}(j))$ is true for every $1\leq j<k$. Next, we assume that $\mathscr{P}(k)$ and $\mathscr{F}(j)=(\mathscr{P}(j), \mathscr{Q}(j))$ for all $1\leq j<k$ are true in order to prove $\mathscr{Q}(k)$.

Cliquifying $G=B\circ H$ at the vertex $w:=v_{12}$, we get
\begin{align*}
    G_w &= (K_{b_1+b_2+h-1}\circ_{b_1+b_2-3}H)\cup_{v_{23}}(B^*\circ H), \text{ where } B^*:=\bigcup_{i=3}^{k}B_i;\\
    G_w \setminus w &= (K_{b_1+b_2+h-2}\circ_{b_1+b_2-3}H)\cup_{v_{23}}(B^*\circ H); \text{ and}\\
    G\setminus w &=H\sqcup (K_{b_1-1}\circ H)\sqcup (\widehat{B}\circ H), \text{ where } \widehat{B}:=(\bigcup_{i=2}^{k}B_i)\setminus \{w\}.
\end{align*}
Clearly, $B^*$ and $\widehat{B}$ are Cohen-Macaulay closed graphs with the number of blocks $<k$. Also, note that 
\begin{align}
    b^*&:=|V(B^*)|=\sum_{i=3}^{k}(b_i-1)+1; \text{ and}\label{eqn11}\\
    \widehat{b}&:=|V(\widehat{B})|=(b_2-2)+\sum_{i=3}^{k}(b_i-1)+1=\sum_{i=2}^{k}(b_i-1).\label{eqn12}
\end{align}
Therefore using Theorem \ref{Thm: Depth and regularity for t-corona with complete graph}, Theorem \ref{Thm: Depth for corona with complete graph} and the induction hypothesis followed by substituting \eqref{eqn10}--\eqref{eqn12}, it follows that
\begin{align*}
    \depth(S/J_{G_w})&=(b_1+b_2+h-1)-(b_1+b_2-3)+(b_1+b_2-3+b^*)\depth(S_H/J_H)\\
    &=h+2+(b-1)\depth(S_H/J_H);\\
    \depth(S/((x_w,y_w)+J_{G_w\setminus w}))&=h+1+(b-1)\depth(S_H/J_H);\\
    \depth(S/((x_w,y_w)+J_{G\setminus w}))&=\depth(S_H/J_H)+(b_1-1)\depth(H)+\widehat{b}\depth(S_H/J_H)\\
    &=b\depth(S_H/J_H); \text{ and }
\end{align*}
\begin{align*}
    \reg(S/J_{G_w}) &=1+(b_1+b_2-3+b^*)\reg(S_H/J_H)=1+(b-1)\reg(S_H/J_H); \\ 
    \reg(S/((x_w,y_w)+J_{G_w\setminus w}))& =1+(b_1+b_2-3+b^*)\reg(S_H/J_H)=1+(b-1)\reg(S_H/J_H); \\
    \reg(S/((x_w,y_w)+J_{G\setminus w}))&=\reg(S_H/J_H)+(b_1-1)\reg(S_H/J_H)+\widehat{b}\reg(S_H/J_H)\\
    &=b\reg(S_H/J_H).
\end{align*}

Since $\reg(S_H/J_H)\geq 2$, so $\reg(S/J_G)=b\reg(S_H/J_H)$ follows from Regularity lemma. By Depth lemma, we have 
$\depth(S/J_G)\leq b\depth(S_H/J_H)$ as $\depth(S_H/J_H)\leq h+1$. Now by the Auslander-Buchsbaum formula, it is enough to show that $\pd(S/J_G)=p\geq 2b+bp_H$. Additionally, we establish that $\beta_{p,p+br_H+1}(S/J_G)$ is an extremal Betti number of $S/J_G$.


Let $p=2b+bp_H$. So Auslander-Buchsbaum formula yields 
\begin{align*}
    \pd(S/J_{G_w}) =2(b+bh)-(h+2+(b-1)(2h-p_H)) =p+h-2-p_H<p;\\
    \pd(S/((x_w,y_w)+J_{G_w\setminus w})) =p+h-1-p_H\leq p \text{ and }
    \pd(S/((x_w,y_w)+J_{G\setminus w})) =p.
\end{align*}

It follows from the induction hypothesis related to extremal Betti numbers that since $r_H\geq 2$, considering the long exact sequence of $\Tor$ \eqref{ohtani-ses} in homological degree $p$, we have the following isomorphism:
\begin{equation}\label{eqn-tor for corona with closed CM}
    \Tor^{S}_{p}\left( \frac{S}{J_G},\mathbf{k}\right)_{p+j} \simeq
    \Tor^{S}_{p}\left(\frac{S}{(x_w,y_w)+J_{G\setminus w}},\mathbf{k}\right)_{p+j} \text{ for }j\geq  br_H.
\end{equation}

Consequently, it turns out that $\beta_{p,p+br_H+1}(S/J_G)$ is an extremal Betti number of $S/J_G$.


Let $Y:=(K_n\circ_\ell H)\cup_{v}(B\circ H)$. By the choice of $v$ as described in the statement $\mathscr{Q}(\ell)$ (see in \eqref{induction 1} above), it is evident that in $Y$ either $v\in V(B_1)\setminus \{v_{12}\}$ or $v\in V(B_k)\setminus \{v_{k-1,k}\}$. Without loss of generality assume that $v\in V(B_1)\setminus \{v_{12}\}$. Observe that
\begin{align*}
    Y_v &= (K_{n+b_1+h-1}\circ_{\ell+b_1-2}H)\cup_{v_{12}}(\widetilde{B}\circ H), \text{ where } \widetilde{B}:=\bigcup_{i=2}^{k}B_i;\\
    Y_v \setminus v &= (K_{n+b_1+h-2}\circ_{\ell+b_1-2}H)\cup_{v_{12}}(\widetilde{B}\circ H); \text{ and}\\
    Y\setminus v &=H\sqcup (K_{n-1}\circ_{\ell} H)\sqcup (\overline{B}\circ H), \text{ where } \overline{B}:=B\setminus \{v\}.
\end{align*}
Clearly, $\widetilde{B}$ and $\overline{B}$ are Cohen-Macaulay closed graphs with the number of blocks exactly $k$. Note that 
\begin{align}
    \widetilde{b}&:=|V(\widetilde{B})|=\sum_{i=2}^{k}(b_i-1)+1; \text{ and}\label{eqn13}\\
    \overline{b}&:=|V(\overline{B})|=b-1.\label{eqn14}
\end{align}
Therefore using Theorem \ref{Thm: Depth and regularity for t-corona with complete graph}, the induction hypothesis and $\mathscr{P}(k)$ followed by substituting \eqref{eqn10}, \eqref{eqn13}, \eqref{eqn14}, it follows that
    \begin{align*}
    \depth(S/J_{G_w})&=(n+b_1+h-1)-(\ell+b_1-2)+(\ell+b_1-2+\widetilde{b})\depth(S_H/J_H)\\
    &=n-\ell+h+1+(\ell+b-1)\depth(S_H/J_H);\\
    \depth(S/((x_w,y_w)+J_{G_w\setminus w}))&=n-\ell+h+(\ell+b-1)\depth(S_H/J_H);\\
    \depth(S/((x_w,y_w)+J_{G\setminus w}))&=\depth(S_H/J_H)+(n-1-\ell+1+\ell\depth(S_H/J_H))+\overline{b}\depth(H)\\
    &=n-\ell+(\ell+\overline{b}+1)\depth(S_H/J_H)\\
    &=n-\ell+(\ell+b)\depth(S_H/J_H).
    \end{align*}
If $\depth(S_H/J_H)\neq h$, then $\mathscr{Q}(k)$ follows immediately from Depth and Regularity Lemmas. For $\depth(S_H/J_H)=h$, the proof of $\mathscr{Q}(k)$ follows exactly along the same line as we did in the proof of $\mathscr{Q}(1)$ using the suitable $\Tor$-sequence.

This proves $\mathscr{F}(\ell)=(\mathscr{P}(\ell),\mathscr{Q}(\ell))$ for all $\ell \in \N$ and the proof of the theorem follows.
\end{proof}
Since a path $P_n$ is an example of a Cohen-Macaulay closed graph with each block being a complete graph on two vertices, which is of special interest, we mention the following straightforward consequence of the Theorem \ref{Thm: Depth for corona with CM closed graph}.
\begin{cor}\label{Cor: Depth for corona with path}
    Let $P_n$ denote the path graph on the vertex set $[n]$ and $H$ be any graph. Consider the graph $G:=P_n\circ H$ for $n \geq 1$. Then
\begin{enumerate}
    \item $\depth(S/J_G)=
\begin{cases}
1+n \cdot \depth(S_H/J_H), &\text{if } H \text{ is complete};\\
n\cdot \depth(S_H/J_H), &\text{otherwise}.
\end{cases}
$
    \item $\reg(S/J_G)=
        \begin{cases}
            n+1, \quad &\text{if } H \text{ is complete};\\
            n\cdot \reg(S_H/J_H), &\text{otherwise}.
        \end{cases}$
\end{enumerate}
\end{cor}

\section*{\bf Acknowledgements} 
The first named author thanks the National Board for Higher Mathematics (NBHM), Department of Atomic Energy, Government of India, for financial support through a Postdoctoral Fellowship.

\section*{\bf Statements and Declarations}

There is no conflict of interest regarding this manuscript. No funding was received for it.

\end{document}